\documentclass[12pt]{amsart}
\usepackage{amsmath,amsfonts,amssymb}
\newtheorem{lemma}{Lemma}[section]
\newtheorem{theorem}[lemma]{Theorem}
\newtheorem{corollary}[lemma]{Corollary}
\newtheorem{proposition}[lemma]{Proposition}

\theoremstyle{definition}

\numberwithin{equation}{section}

\DeclareMathOperator{\Be}{Be}
\DeclareMathOperator{\BI}{BI}

\DeclareMathOperator{\sym}{sym}

\begin{document}

\newcommand{\ZZ}{\mathbb{Z}}
\newcommand{\ZZd}{\mathbb{Z}^{d}}
\newcommand{\RR}{\mathbb{R}}
\newcommand{\RRd}{\mathbb{R}^{d}}
\newcommand{\PP}{\mathbb{P}}
\newcommand{\QQ}{\mathbb{Q}}
\newcommand{\EE}{\mathbb{E}}
\newcommand{\mB}{\mathcal{B}}
\newcommand{\mC}{\mathcal{C}}
\newcommand{\mD}{\mathcal{D}}
\newcommand{\mE}{\mathcal{E}}
\newcommand{\mF}{\mathcal{F}}
\newcommand{\mG}{\mathcal{G}}
\newcommand{\mH}{\mathcal{H}}
\newcommand{\mI}{\mathcal{I}}
\newcommand{\mJ}{\mathcal{J}}
\newcommand{\mL}{\mathcal{L}}
\newcommand{\mM}{\mathcal{M}}
\newcommand{\mkN}{\mathfrak{N}}
\newcommand{\mO}{\mathcal{O}}
\newcommand{\mQ}{\mathcal{Q}}
\newcommand{\mR}{\mathcal{R}}
\newcommand{\mS}{\mathcal{S}}
\newcommand{\mT}{\mathcal{T}}
\newcommand{\mU}{\mathcal{U}}
\newcommand{\mW}{\mathcal{W}}
\newcommand{\mY}{\mathcal{Y}}
\newcommand{\bs}{\backslash}
\newcommand{\half}{\frac{1}{2}}
\newcommand{\fN}{\frac{1}{N}}
\newcommand{\bx}{\mathbf{x}}
\newcommand{\by}{\mathbf{y}}
\newcommand{\bV}{\mathbf{V}}
\newcommand{\olV}{\overline{V}}
\newcommand{\olm}{\overline{m}}
\newcommand{\olf}{\overline{f}}
\newcommand{\olS}{\overline{S}}
\newcommand{\ep}{\epsilon}
\newcommand{\DM}{\Delta_{\max}}
\newcommand{\Dm}{\Delta_{\min}}
\newcommand{\tn}{\tilde{n}}
\newcommand{\tc}{\tilde{\chi}}
\newcommand{\tP}{\tilde{P}}

\title[Bessel-like random walks]{Excursions and local limit theorems for Bessel-like random walks}
\author{Kenneth S. Alexander}
\address{Department of Mathematics KAP 108\\
University of Southern California\\
Los Angeles, CA  90089-2532 USA}
\email{alexandr@usc.edu}
\thanks{This research was supported by NSF grant DMS-0804934.}


\begin{abstract}
We consider reflecting random walks on the nonnegative integers with drift of order $1/x$ at height $x$.  We establish explicit asymptotics for various probabilities associated to such walks, including the distribution of the hitting time of 0 and first return time to 0, and the probability of being at a given height $k$ at time $n$ (uniformly in a large range of $k$.)  In particular, for drift of form $-\delta/2x + o(1/x)$ with $\delta>-1$, we show that the probability of a first return to 0 at time $n$ is asymptotically $n^{-c}\varphi(n)$, where $c = (3+\delta)/2$ and $\varphi$ is a slowly varying function given in terms of the $o(1/x)$ terms.
\end{abstract}
\maketitle

\section{Introduction}
We consider random walks on $\ZZ_+ = \{0,1,2,\dots\}$, reflecting at 0, with steps $\pm 1$ and transition probabilities of the form
\begin{equation} \label{Bessellike}
  p(x,x+1) = p_x =  \half\left( 1 - \frac{\delta}{2x} + o\left( \frac{1}{x} \right) \right) \ \text{as } x \to \infty, \quad p(x,x-1) = q_x
    = 1-p_x,
  \end{equation}
for $x \geq 1$.  We call such processes \emph{Bessel-like walks}, as their drift is asymptotically the same as that of a Bessel process of (possibly negative) dimension $1-\delta$.   We call $\delta$ the \emph{drift parameter}.  Bessel-like walks are a special case of what is called the Lamperti problem---random walks with asymptotically zero drift.  A Bessel-like walk is recurrent if $\delta > -1$, positive recurrent if $\delta>1$, and transient if $\delta < -1$; for $\delta=-1$ recurrence or transience depends on the $o(1/x)$ terms.  Here we consider the recurrent case, with primary focus on $\delta>-1$, as the case $\delta=-1$ has additional complexities which weaken our results.  Bessel-like walks arise for example when (reflecting) symmetric simple random walk (SSRW) is modified by a potential proportional to $\log x$.  

Bessel-like walks have been extensively studied since the 1950's.  Hodges and Rosenblatt \cite{HR53} gave conditions for finiteness of moments of certain passage times, and Lamperti \cite{La62} established a functional central limit theorem (with non-normal limit marginals) for $\delta<1$; for $-1<\delta<1$ our Theorem \ref{location} below is a local version of his CLT.  
In \cite{La63} Lamperti related the first and second moments of the step distribution to finiteness of integer moments of first-return-time distributions.  He worked with a wider class of Markov chains  with drift of order $1/x$, showing in particular that for return times of Bessel-like walks, moments of order less than $\kappa = (1+\delta)/2$ are finite while those of order greater than $\kappa$ are infinite.  Lamperti's results were generalized and extended to noninteger moments in \cite{AI97},  \cite{AIM96}, and to expected values of more general functions of return times in \cite{AI99}. 
``Upper and lower'' local limit theorems were established in \cite{MP95} for certain positive recurrent processes which include our $\delta>1$.
Bounds for the growth rate of processes with drift of order $1/x$ were given in \cite{MVW08}, and the domain of attraction of the excursion length distribution was examined in \cite{Fa73}.

Karlin and McGregor (\cite{KM57a}, \cite{KM57b}, \cite{KM59}) showed that, for general birth-death processes, many quantities of interest could be expressed in terms of a family of polynomials orthogonal with respect to a measure on $[-1,1]$.  This measure can in principle be calculated (see Section 8 of \cite{KM57b}) but not concretely enough, apparently, for some computations we will do here.  An exception is the case of $p_x = \half(1 - \frac{\delta}    {2x+\delta}    )$ considered in \cite{De01} (for $\delta=1$) and \cite{DDH08}; we will call this the \emph{rational-form case}.  Birth-death processes dual to the rational form case were considered in \cite{Ro66}.  Further results for birth-death processes via the Karlin-McGregor representation are in \cite{CV98}, \cite{Fa81}.  

Our interest in Bessel-like walks originates in statistical physics.  These walks were used in \cite{DDH09} in a model of wetting.  Additionally, in polymer pinning models of the type studied in \cite{Gi07} and the references therein, there is an underlying Markov chain which interacts with a potential at times of returns to 0.  The location of the $i$th monomer is given by the state of the chain at time $i$.  There may be quenched disorder, in the form of random variation in the potential as a function of the time of the return.  Let $\tau_0$ denote the return time to 0 for the Markov chain started at 0.  For many models of interest, e.g. SSRW on $\ZZ^d$, the distribution of $\tau_0$ for the underlying Markov chain has a power-law tail:  
\begin{equation} \label{powerlaw}
  P(\tau_0 = n) = n^{-c}\varphi(n)
  \end{equation}
for some $c \geq 1$ and slowly varying $\varphi$.  Considering even $n$, for $d=1$ one has $c=3/2$ and $\varphi(n)$ converging to $\sqrt{2/\pi}$; for $d=2$ one has $c=1$ and $\varphi(n)$ proportional to $(\log n)^{-2}$ \cite{JP72}; for $d \geq 3$ one has $c=d/2$ and $\varphi(n)$ asymptotically constant.  In general the value of $c$ is central to the critical behavior of the polymer with the presence of the disorder altering the critical behavior for $c>3/2$ but not for $c<3/2$ (\cite{Al08},\cite{AZ08},\cite{GT05}.)  In the ``marginal'' case $c=3/2$, the slowly varying function $\varphi$ determines whether the disorder has such an effect \cite{GLT09}. As we will see, for Bessel-like walks, \eqref{powerlaw} holds in the approximate sense that
\begin{equation} \label{powerlaw2}
  P(\tau_0 = n) \sim n^{-c}\varphi(n) \quad \text{as } n \to \infty,
  \end{equation} 
with $c=(3+\delta)/2$ and $\varphi(n)$ determined explicitly by the $o(1/x)$ terms.  Here $\sim$ means the ratio converges to 1.  Thus Bessel-like walks provide a single family of Markov chains in $(1+1)$-dimensional space-time in which \eqref{powerlaw} can be realized (at least asymptotically) for arbitrary $c$ and $\varphi$.

A related model is the directed polymer in a random medium (DPRM), in which the underlying Markov chain is generally taken to be SSRW on $\ZZ^d$ and the polymer encounters a random potential at every site, not just the special site 0.  The DPRM has been studied in both the physics literature (see the survey \cite{HZ95}) and the mathematics literature (see e.g.~\cite{CH02}, \cite{CY07}, \cite{La10}.)  In place of SSRW, one could use a Markov chain on $\ZZ^d$ in which each coordinate is an independent Bessel-like walk.  In this manner one could study the effect on the DPRM of the behavior \eqref{powerlaw2}, or more broadly, study the effect of the drift present in the Bessel-like walk.  As with the pinning model, via Bessel-like walks, all drifts and all tail exponents $c$ (not just the half-integer values occurring for SSRW) can be studied using the same space of trajectories.  This will be pursued in future work.

For the DPRM, an essential feature is the overlap, that is, the value
\[
  \sum_{i=1}^N \delta_{\{X_i=X_i'\}},
  \]
where $\{X_i\}, \{X_i'\}$ are two independent copies of the Markov chain; see (\cite{CH02}, \cite{CY07}, \cite{La10}.)  To determine the typical behavior of the overlap one should know the probabilities $P(X_i=y), y \in \ZZ^d$, as precisely as possible, with as much uniformity in $y$ as possible..

For this paper we thus have two goals: given the transition probabilities $p_x,q_x$ of a Bessel-like walk, determine 
\begin{itemize}
\item[(i)] the value $c$ and slowly varying function $\varphi$ for which \eqref{powerlaw2} holds, and 
\item[(ii)] the probabilities $P(X_i=y), y \in \ZZ$, asymptotically as $i \to \infty$, as uniformly in $y$ as possible.
\end{itemize}
We will not make use of the methods of Karlin and McGregor (\cite{KM57a}, \cite{KM57b}, \cite{KM59}) due to the difficulty of calculating the measure explicitly enough, and obtaining the desired uniformity in $y$.  Instead we take a more probabilistic approach, comparing the Bessel-like walk to a Bessel process with the same drift, while the walk is at high enough heights.  This leads to estimates of probabilities of form $P(\tau_0 \in [a,b])$ when $a/b$ is bounded away from 1.  Then to obtain \eqref{powerlaw2} we use special coupling properties of birth-death processes which force regularity on the sequence $\{P(\tau_0 = n), n \geq 1\}$.  These properties, given in Lemma \ref{latticepath} and Corollary \ref{latticepath2}, may be of some independent interest.

\section{Main Results}
Consider a Bessel-like random walk $\{X_n\}$ on the nonnegative integers with drift parameter $\delta \geq -1$, with transition probabilities $p_x = p(x,x+1),q_x=p(x,x-1)=1-p_x$.  The walk is reflecting, i.e. $p_0=1$.  We assume uniform ellipticity:  there exists $\ep>0$ for which 
\begin{equation} \label{elliptic}
  p_x,q_x \in [\ep,1-\ep] \quad \text{for all } x \geq 1.
  \end{equation}
Define $R_x$ by 
\begin{equation} \label{pxRx}
  p_x = \half\left( 1 - \frac{\delta}{2x} + \frac{R_x}{2} \right),
  \end{equation}
where $R_x = o(1/x)$.  Note that in the rational-form case we have
\[
  R_x = \frac{\delta^2}{2x^2} + O\left( \frac{1}{x^3} \right).
  \]
The drift at $x$ is
\[
  p_x - q_x = 2p_x-1 = -\frac{\delta}{2x} + \frac{R_x}{2}.
  \]
Let $\lambda_0 = 1, M_0=0$ and for $x \geq 1$,
\[
  \lambda_x = \prod_{k=1}^x \frac{q_k}{p_k}, \quad M_x = \sum_{k=0}^{x-1} \lambda_k, 
    \quad L(x) = \exp\left( R_1 + \dots + R_x\right).
  \]
$M_x$ is the scale function.  Note $M_1=1$, and $M_{X_n \wedge \tau_0}$ is a martingale.  It is easily checked that the assumption $R_x = o(1/x)$ ensures $L$ is slowly varying.  By linearly interpolating between integers, we can extend $L$ to a function on $[1,\infty)$ which is still slowly varying.  Let $\tau_j$ be the hitting time of $j \in \ZZ_+$, let $P_j$ denote probability for the walk started from height $j$ and let 
\begin{equation} \label{Hdef}
  H = \max \{X_i: i \leq \tau_0\}
  \end{equation}
be the height of an excursion from 0.  From the martingale property we have
\begin{equation} \label{height}
  P_0(H\geq h) = P_1(\tau_h<\tau_0) = \frac{M_1}{M_h} 
  \end{equation}
so since $M_1=1$,
\[
  P_0(H=h) =  \frac{M_1}{M_h} -  \frac{M_1}{M_{h+1}} = \frac{\lambda_h}{M_h M_{h+1}}.
  \]
In place of $\delta$, a more convenient parameter is often
\[
  \kappa = \frac{1+\delta}{2} \geq 0.
  \]
We have
\[  
  \frac{p_x}{q_x} = 1 - \frac{\delta}{x} + R_x + O\left( \frac{1}{x^2} \right),
  \]
and hence
\begin{equation} \label{lambda}
  \lambda_x \sim K_0x^{2\kappa-1}L(x)^{-1} \ \text{as $x \to \infty$, for some } K_0>0,
  \end{equation}
so for $\kappa>0$,
\begin{equation} \label{Mxapprox}
  M_x \sim \frac{K_0}{2\kappa} x^{2\kappa} L(x)^{-1}.
  \end{equation} 
Our assumption of recurrence is equivalent to $M_x \to \infty$.

Define the slowly varying function
\[
  \nu(n) = \sum_{l \leq n,\ l \text{ even}} \frac{1}{lL(\sqrt{l})}.
  \]

Throughout the paper, $K_0, K_1,\dots$ are constants which depend only on $\{p_x, x \geq 1\}$, except as noted; for example, $K_i(\theta,\chi)$ means that $K_i$ depends on some previously-specified $\theta$ and $\chi$.  Further, to avoid the notational clutter of pervasive integer-part symbols, we tacitly assume that all indices which appear are integers, as may be arranged by slightly modifying various arbitrarily-chosen constants, or more simply by mentally inserting the integer-part symbol as needed.

\begin{theorem} \label{tau0tail}
Assume \eqref{pxRx} and \eqref{elliptic}.  For $\delta>-1$, 
\begin{equation} \label{exctail}
  P_0(\tau_0 \geq n) \sim \frac{2^{1-\kappa}}{K_0 \Gamma(\kappa)} n^{-\kappa} L(\sqrt{n}) \quad \text{as } n \to \infty,
  \end{equation}
and for $n$ even,
\begin{equation} \label{excpoint}
  P_0(\tau_0 = n) \sim \frac{2^{2-\kappa}\kappa}{K_0 \Gamma(\kappa)} n^{-(\kappa+1)} L(\sqrt{n}).
  \end{equation}
For $\delta = -1$, assuming recurrence (i.e. $M_x \to \infty$ as $x \to \infty$),
\begin{equation} \label{exctailsv}
  P_0(\tau_0 \geq n) \sim \frac{1}{K_0 \nu(n)}.
  \end{equation}
\end{theorem}

For the case of SSRW, in contrast to \eqref{excpoint}, the excursion length distribution  is easily given exactly \cite{Fe68}:  for $n$ even,
\[
  P_0(\tau_0=n) = \frac{1}{n-1} {n \choose n/2} 2^{-n} \sim \frac{1}{2\sqrt{\pi}}n^{-3/2}.
  \]

By \eqref{exctail} we have for fixed $\eta \in (0,1)$ that
\begin{equation} \label{exctail2}
  P_0\big((1-\eta)n \leq \tau_0 \leq (1+\eta) n \big) 
    \sim \frac{ 2^{2-\kappa} }{ K_0 \Gamma(\kappa) } \eta \Upsilon(\eta) n^{-\kappa}L(\sqrt{n}),
  \end{equation}
where
\begin{equation} \label{psidef}
  \Upsilon(\eta) = \frac{1}{2\eta} \left( (1-\eta)^{-\kappa} - (1+\eta)^{-\kappa} \right) \to \kappa \quad \text{as } \eta \to 0.
  \end{equation}
Heuristically, one expects that conditionally on the event on the left side of \eqref{exctail2}, $\tau_0$ should be approximately uniform over even numbers in the interval $[(1-\eta)n,(1+\eta) n]$, leading to \eqref{excpoint}.  The precise statement we use is Lemma \ref{convexity}.

It follows from \eqref{height}, \eqref{Mxapprox} and Theorem \ref{tau0tail} that $\tau_0$ and $H^2$ have asymptotically the same tail, to within a constant:
\begin{equation} \label{sametail}
  P_0(H^2 \geq n) \sim 2^\kappa \kappa \Gamma(\kappa) P_0(\tau_0 \geq n) 
    \sim P_0\left( 2 (\kappa \Gamma(\kappa))^{1/\kappa} \tau_0 \geq n \right) \quad \text{as } n \to \infty.
  \end{equation}
This says roughly that the typical height of an excursion becomes a large multiple of the square root of its length (i.e. duration), as $\kappa$ grows,   meaning the downward drift becomes stronger.  In this sense the random walk climbs higher to avoid the strong drift.

By reversing paths we see that
\begin{equation} \label{reverse}
  P_k(X_n = 0) = p_k \lambda_k P_0(X_n = k).
  \end{equation}
Hence to obtain an approximation for $P_0(X_n=k)$, we need an approximation for $P_k(X_n=0)$, and for that we first need an approximation for $P_k(\tau_0 = m)$.  In this context, keeping in mind the similarity between $\tau_0$ and $H^2$, for a given constant $\chi<1$ we say that a starting (or ending) height $k$ is \emph{low} if $k < \sqrt{\chi m}$, \emph{midrange} if $\sqrt{m\chi} \leq k \leq \sqrt{m/\chi}$ and \emph{high} if $k > \sqrt{m/\chi}$.

\begin{theorem} \label{hittime}
Suppose $\delta>-1$.  Given $\theta>0$, for $\chi>0$ sufficiently small, there exists $m_0(\theta,\chi)$ as follows.  For all $m \geq m_0$ and $1 \leq k < \sqrt{\chi m}$ (low starting heights) with $m-k$ even,
\begin{align} \label{hitapprox}
  (1-\theta) \frac{2^{2-\kappa}\kappa}{K_0\Gamma(\kappa)} m^{-(1+\kappa)} L(\sqrt{m}) M_k
    &\leq P_k(\tau_0 =m) \\
  &\leq (1+\theta) \frac{2^{2-\kappa}\kappa}{K_0\Gamma(\kappa)} m^{-(1+\kappa)} L(\sqrt{m}) M_k. \notag
  \end{align}
For all $\sqrt{m\chi} \leq k \leq \sqrt{m/\chi}$ (midrange starting heights) with $m-k$ even, 
\begin{equation} \label{hitapprox2}
  (1-\theta)\frac{2}{\Gamma(\kappa)m} \left( \frac{k^2}{2m} \right)^\kappa e^{-k^2/2m}
    \leq P_k(\tau_0 =m) 
    \leq (1+\theta)\frac{2}{\Gamma(\kappa)m} \left( \frac{k^2}{2m} \right)^\kappa e^{-k^2/2m}. 
  \end{equation}
For all $k > \sqrt{m/\chi}$ (high starting heights) with $m-k$ even,
\begin{equation} \label{hitapprox3}
  P_k(\tau_0 =m) \leq \frac{1}{m} e^{-k^2/8m}.
  \end{equation}
  \end{theorem}
  
In general, for high starting heights, as in \eqref{hitapprox3} we accept upper bounds, rather than sharp approximations as in \eqref{hitapprox} and \eqref{hitapprox2}.

Note that by \eqref{Mxapprox}, when $k$ is large \eqref{hitapprox} and \eqref{hitapprox2} differ only in the factor $e^{-k^2/2m}$, which is near 1 for low starting heights.  (Here ``large'' does not depend on $m$.)  Further, by \eqref{excpoint}, one can replace \eqref{hitapprox} with 
\begin{align} \label{hitapprox4}
  (1-\theta) P_0(\tau_0 = m) M_k \leq P_k(\tau_0 =m) \leq (1+\theta)P_0(\tau_0 = m) M_k. 
  \end{align}
  
We will see below that the left and right sides of \eqref{hitapprox2} represent approximately the probabilities for a Bessel process, with the same drift parameter $\delta$ and starting height $k$, to hit 0 in $[m-1,m+1]$.  But the Bessel approximation is not necessarily valid for low starting heights, where \eqref{hitapprox} holds, because the analog of $M_k$ for the Bessel process may be quite different from its value for the Bessel-like RW, and because $L(\sqrt{m})/L(k)$ need not be near 1, whereas the analog of $L(\cdot)$ for the Bessel process is a constant.  Even if a RW has asymptotically constant $L(\cdot)$, the constant $K_0$ may be different from the related Bessel case.

From \eqref{hitapprox2}, for midrange starting heights the distribution of $\tau_0$ is nearly the same as for the approximating Bessel process.  For low starting heights, this is not true in general---the Bessel-like RW in this case will typically climb to a height of order $\sqrt{m}$ for paths with $\tau_0 = m$, and this climb is what is affected by the dissimilarity between the two processes, as reflected in the errors $R_x$.

If $\delta>1$ (i.e. $\kappa>1$), or if $\delta=1$ and $E_0(\tau_0) < \infty$, then
\begin{equation} \label{finitemean3}
  P_0(X_n=0) \to \frac{2}{E_0(\tau_0)} \quad \text{as } n \to \infty \quad (n \text{ even}),
  \end{equation}
and of course when it is finite, $E_0(\tau_0)$ can be expressed explicitly in terms of the transition probabilities $p_x$ and $q_x$, by using reversibility.
If $-1<\delta<1$ (i.e. $0<\kappa<1$), then by \eqref{excpoint} and a result of Doney \cite{Do97},
\begin{equation} \label{infinitemean3}
  P_0(X_n=0) \sim \frac{2^\kappa K_0}{\Gamma(1-\kappa) } n^{-(1-\kappa)} L(\sqrt{n})^{-1} \quad (n \text{ even}),
  \end{equation}
and if $\delta=1$ (i.e. $\kappa=1$) with $E_0(\tau_0) = \infty$, then by \eqref{excpoint} and a result of Erickson \cite{Er70},
\begin{equation} \label{c2case3}
  P_0(X_n=0) \sim \frac{2}{\mu_0(n)} \quad (n \text{ even}),
  \end{equation}
where $\mu_0(n)$ is the truncated mean:
\[
  \mu_0(n) = \sum_{l=1}^n l P_0(\tau_0 = l) \sim \frac{2}{K_0} \sum_{l\leq n,\ l\text{ even}} \frac{L(\sqrt{l})}{l},
  \]
which is a slowly varying function.
  
The next theorem, approximating the left side of \eqref{reverse}, is based on Theorem \ref{hittime} and \eqref{finitemean3}---\eqref{c2case3}, together with the fact that
\begin{equation} \label{renewal}
  P_k(X_n=0) = \sum_{j=0}^n P_k(\tau_0=n-j)P_0(X_j=0).
  \end{equation}

\begin{theorem} \label{distrib}
Given $\theta>0$, for $\chi$ sufficiently small there exists $n_0(\theta,\chi)$ such that for all $n \geq n_0$, the following hold.

(i) For $k < \sqrt{\chi n}$ (low starting heights) with $n-k$ even, 
\begin{equation} \label{lowheight}
  (1-\theta)P_0(X_{\tn} = 0) \leq P_k(X_n = 0) \leq (1+\theta)P_0(X_{\tn} = 0),
  \end{equation}
where $\tn = n$ if $n$ is even, $\tn = n+1$ if $n$ is odd.
  
(ii) If $E_0(\tau_0) < \infty$ (which is always true for $\delta>1$), then for $\sqrt{n\chi} \leq k \leq \sqrt{n/\chi}$ (midrange starting heights) with $n-k$ even,
\begin{align} \label{finitemean2}
  \frac{2-\theta}{E_0(\tau_0)} \int_{k^2/2n}^\infty \frac{1}{\Gamma(\kappa)} u^{\kappa-1} e^{-u}\ du &\leq P_k(X_n = 0) \\
  &\leq \frac{2+\theta}{E_0(\tau_0)} \int_{k^2/2n}^\infty \frac{1}{\Gamma(\kappa)} u^{\kappa-1} e^{-u}\ du, \notag
  \end{align}
and for $k > \sqrt{n/\chi}$ (high starting heights) with $n-k$ even,
\begin{align} \label{finitemean4}
  P_k(X_n = 0) &\leq \frac{8}{E_0(\tau_0)} e^{-k^2/8n}.
\end{align}
  
(iii) If $-1<\delta<1$, then for $\sqrt{n\chi} \leq k \leq \sqrt{n/\chi}$ (midrange starting heights) with $n-k$ even,
\begin{align} \label{infinitemean}
  \frac{ (1-\theta)2^\kappa K_0 }{ \Gamma(1-\kappa) } n^{-(1-\kappa)} L(\sqrt{n})^{-1} e^{-k^2/2n}
    &\leq P_k(X_n=0) \\
  &\leq \frac{ (1+\theta)2^\kappa K_0 }{ \Gamma(1-\kappa) } 
    n^{-(1-\kappa)} L(\sqrt{n})^{-1} e^{-k^2/2n}, \notag
  \end{align}
and there exists $K_{1}(\kappa)$ such that for $k > \sqrt{n/\chi}$ (high starting heights) with $n-k$ even,
\begin{align} \label{infinitemean3a}
  P_k(X_n = 0) &\leq K_{1} e^{-k^2/8n} n^{-(1-\kappa)} L(\sqrt{n})^{-1}.
\end{align}
  
(iv) If $\delta=1$ and $E_0(\tau_0) = \infty$, then for $\sqrt{n\chi} \leq k \leq \sqrt{n/\chi}$ (midrange starting heights) with $n-k$ even,
\begin{align} \label{borderline}
  \frac{2-\theta}{\mu_0(n)} \int_{k^2/2n}^\infty \frac{1}{\Gamma(\kappa)} u^{\kappa-1} e^{-u}\ du &\leq P_k(X_n = 0) \\
  &\leq \frac{2+\theta}{\mu_0(n)} \int_{k^2/2n}^\infty \frac{1}{\Gamma(\kappa)} u^{\kappa-1} e^{-u}\ du, \notag
  \end{align}
and for $k > \sqrt{n/\chi}$ (high starting heights) with $n-k$ even,
\begin{align} \label{borderline3}
  P_k(X_n = 0) &\leq \frac{8}{\mu_0(n)} e^{-k^2/8n}.
\end{align}
  \end{theorem}
  
From \cite{GJY03}, the integral that appears in \eqref{finitemean2} and \eqref{borderline} is the probability that the approximating Bessel process started at $k$ hits 0 by time $n$. 

We may of course replace $P_0(X_n=0)$ with the appropriate approximation from \eqref{finitemean3}---\eqref{c2case3}, in \eqref{lowheight}.

We now combine \eqref{reverse} with Theorem \ref{distrib} to approximate the left side of \eqref{reverse}.

\begin{theorem} \label{location}
Given $\theta>0$, for $\chi>0$ sufficiently small, there exists $n_0(\theta,\chi)$ such that for all $n \geq n_0$, the following hold.

(i) For $1 \leq k < \sqrt{\chi n}$ (low ending heights) with $n-k$ even, 
\begin{equation} \label{lowheight5}
  \frac{1-\theta}{\lambda_k p_k} P_0(X_n = 0) \leq P_0(X_n = k) \leq \frac{1+\theta}{\lambda_k p_k} P_0(X_n = 0).
  \end{equation}
  
(ii) If $E_0(\tau_0) < \infty$ (which is always true for $\delta>1$), then for $\sqrt{n\chi} \leq k \leq \sqrt{n/\chi}$ (midrange ending heights) with $n-k$ even,
\begin{align} \label{finitemean5}
  (1-\theta)&\frac{4}{K_0 E_0(\tau_0)} k^{1-2\kappa} L(k) \int_{k^2/2n}^\infty \frac{1}{\Gamma(\kappa)} u^{\kappa-1} e^{-u}\ du \\
  &\leq P_0(X_n = k) \leq (1+\theta)\frac{4}{K_0 E_0(\tau_0)} k^{1-2\kappa} L(k) \int_{k^2/2n}^\infty \frac{1}{\Gamma(\kappa)} 
    u^{\kappa-1} e^{-u}\ du, \notag
  \end{align}
and for $k > \sqrt{n/\chi}$ (high ending heights) with $n-k$ even,
\begin{equation} \label{finitemean6}
  P_0(X_n = k) \leq\frac{32}{K_0 E_0(\tau_0)} k^{1-2\kappa} L(k) e^{-k^2/8n}.
\end{equation}
  
(iii) If $-1<\delta<1$, then for $\sqrt{n\chi} \leq k \leq \sqrt{n/\chi}$ (midrange ending heights) with $n-k$ even,
\begin{align} \label{infinitemean5}
  &(1-\theta)\frac{ 2^{\kappa+1}}{ \Gamma(1-\kappa) } \left( \frac{k}{\sqrt{n}} 
    \right)^{1-2\kappa}  e^{-k^2/2n} n^{-1/2} \\
  &\qquad \qquad \leq P_0(X_n = k) \leq (1+\theta) \frac{ 2^{\kappa+1}}{ \Gamma(1-\kappa) } 
     \left( \frac{k}{\sqrt{n}} \right)^{1-2\kappa} e^{-k^2/2n} n^{-1/2}, \notag
  \end{align}
and for $k > \sqrt{n/\chi}$ (high ending heights) with $n-k$ even, for $K_{1}$ of \eqref{infinitemean3a},
\begin{equation} \label{infinitemean6}
  P_0(X_n = k) \leq \frac{4K_{1}}{K_0} e^{-k^2/8n} n^{-1/2}.
\end{equation}
  
(iv) If $\delta=1$ and $E_0(\tau_0) = \infty$, then for $\sqrt{n\chi} \leq k \leq \sqrt{n/\chi}$ (midrange ending heights) with $n-k$ even,
\begin{align} \label{borderline5}
  (1-\theta)&\frac{4}{K_0 \mu_0(n)} \frac{L(k)}{k} \int_{k^2/2n}^\infty \frac{1}{\Gamma(\kappa)} u^{\kappa-1} e^{-u}\ du \\
  &\leq P_0(X_n = k) \leq (1+\theta)\frac{4}{K_0 \mu_0(n)} \frac{L(k)}{k} \int_{k^2/2n}^\infty \frac{1}{\Gamma(\kappa)} u^{\kappa-1} e^{-u}\ du, \notag
  \end{align}
and for $k > \sqrt{n/\chi}$ (high ending heights) with $n-k$ even,
\begin{equation} \label{borderline6}
  P_0(X_n = k) \leq\frac{44}{K_0\mu_0(n)} \frac{L(k)}{k} e^{-k^2/8n}.
\end{equation}

\end{theorem}

A version of \eqref{infinitemean5} for the RW dual to the rational-form case, with $\delta=-1$, was proved in \cite{Ro66}, with the statement that the proof works for general $\delta<1$.

For large $k$ we can use the approximation \eqref{lambda} in \eqref{lowheight5}.  For example, in the case $-1<\delta<1$, there exists $k_1(\theta)$ such that for $n \geq n_0$ and $k_1 \leq k < \sqrt{\chi n}$ we have
\begin{align} \label{uselambda}
  (1-\theta)& \frac{2^{2-\kappa}}{\Gamma(1-\kappa)} n^{-(1-\kappa)} k^{-\delta} \frac{L(k)}{ L(\sqrt{n}) } \\
  &\leq P_0(X_n = k) \leq (1+\theta) \frac{2^{2-\kappa}}{\Gamma(1-\kappa)} n^{-(1-\kappa)} k^{-\delta} \frac{L(k)}{ L(\sqrt{n}) }. \notag
  \end{align}
  
We can use Theorem \ref{location} to approximately describe the distribution of $X_n$ only because its statement gives uniformity in $k$.  This requires uniformity in $k$ in Theorems \ref{hittime} and \ref{distrib}, which points us toward our probabilistic approach.

The factors 8 in the exponent in \eqref{finitemean6}, \eqref{infinitemean6} and \eqref{borderline6} is not sharp.  For $-2<\delta<0$, bounds on tail (not point) probabilities with sharper exponents are established in \cite{BRS71}.

We are unable to extend our results to random walks with drift which is asymptotically 0 but not of order $1/x$, because we rely on known properties of the Bessel process.

\section{Coupling} \label{Coupling}

Let us consider the random walk with steps $\pm 1$ imbedded in a Bessel process $Y_t \geq 0$ with drift $-\delta/2Y_t$:
\[
  dY_t = -\frac{\delta}{2Y_t}\ dt + dB_t,
  \]
where $B_t$ is Brownian motion.  (We need only consider this process until the time, if any, that it hits 0, which avoids certain technical complications.)  The imbedded walk is defined in the standard way:  we start both the RW and the Bessel process at the same integer height $k$.  The first step of the RW is to $k \pm 1$, whichever the Bessel process hits first, at some time $S_1$.  The second step is to $Y_{S_1} \pm 1$, whichever the Bessel process hits first starting from time $S_1$, and so on.

Let $g(x) = x^{1+\delta}$; then $g(Y_t)$ is a martingale, in fact a time change of Brownian motion (see \cite{RY91}.)  Write $P^{\Be}$ for probability for the Bessel process, $P^{\BI}$ for the imbedded RW and $P^{\sym}$ for symmetric simple random walk (not reflecting at 0.)  For the imbedded RW, for $x \geq 1$, the downward transition probability is
\[
  q_x^{\BI} = P_x^{\Be}(\tau_{x-1} < \tau_{x+1}) = \frac{g(x+1) - g(x)}{g(x+1)-g(x-1)}
    = \half\left(1 + \frac{\delta}{2x} + \frac{\delta^2(1-\delta)}{12x^3} + O\left(\frac{1}{x^4}\right) \right)
  \]
so the corresponding value of $R_x$ is
\[
  R_x^{\BI} = -\frac{\delta^2(1-\delta)}{6x^3} + O\left(\frac{1}{x^4}\right).
  \]
We write $\{X_n\}$, $\{X_n^{\BI}\}$ and $\{X_n^{\sym}\}$ for the Bessel-like RW, imbedded RW, and symmetric simple RW, respectively, and $\tau_j,\tau_j^{\BI},\tau_j^{\sym}$ for the corresponding hitting times.

Here is a special construction of $\{X_n\}$ that couples it to $\{X_n^{\sym}\}$, when $p_x \leq q_x$ for all $x$.  (A similar construction works in case $p_x \geq q_x$ for all $x$.)  Let $\xi_0,\xi_1,\dots$ be i.i.d. uniform in [0,1].  For each $i \geq 0$ we have an alarm independent of $\xi_i$.  If $X_i=x$, the alarm sounds with probability $q_x - p_x = \frac{\delta}{2x} - \frac{R_x}{2}$.  If there is no alarm, $X_{i+1} = x+1$ if $\xi_i > 1/2$, and $X_{i+1}=x-1$ if $\xi_i \leq 1/2$.  If the alarm sounds, then $X_{i+1}=x-1$, regardless of $\xi_i$.  $\{X_n^{\sym}\}$ ignores the alarm and always takes its step according to $\xi_i$.

A second special construction, coupling $\{X_n\}$ to $\{X_n^{\BI}\}$, is as follows; a related coupling appears in \cite{CFR08}.  If $X_i=x$, the alarm sounds independently with probability $a(x)$ given by
\[
  a(x) = \begin{cases} \frac{p_x - p_x^{\BI}}{q_x^{\BI}} = \frac{R_x}{2} + \frac{\delta^2(1-\delta)}{12x^3} 
    + O\left( \frac{|R_x|}{x} + \frac{1}{x^4} \right) \quad &\text{if } p_x \geq p_x^{\BI}, \\
    \frac{q_x - q_x^{\BI}}{p_x^{\BI}} = -\frac{R_x}{2} - \frac{\delta^2(1-\delta)}{12x^3} 
      + O\left( \frac{|R_x|}{x} + \frac{1}{x^4} \right)  \quad &\text{if } p_x < p_x^{\BI}.
    \end{cases}
  \]
Whenever the alarm sounds, $\{X_i\}$ takes a step up in the case $p_x \geq p_x^{\BI}$, and down in the case $p_x < p_x^{\BI}$.  If there is no alarm, $\{X_n\}$ goes up if $\xi_i > q_x^{\BI}$ and down if $\xi_i \leq q_x^{\BI}$.  By contrast, $\{X_n^{\BI}\}$ ignores the alarm and always takes its step according to $\xi_i$.  Under this construction, if $p_x \geq p_x^{\BI}$, the probability of an up step for $\{X_i\}$ from $x$ is 
\[
  (1-a(x))p_x^{\BI} + a(x) \cdot 1 = p_x,
  \]
and if $p_x < p_x^{\BI}$, the probability of a down step for $\{X_i\}$ is
\[
  (1-a(x))q_x^{\BI} + a(x) \cdot 1 = q_x,
  \]
which shows that this second construction does indeed couple $\{X_n\}$ to $\{X_n^{\BI}\}$.  Note that in the second construction, unlike the first, the frequency of alarms is $o(1/x)$.  The coupling to $\{X_n^{\BI}\}$ is more complicated because the transition probabilities for  $\{X_n^{\BI}\}$ depend on location.  Even when no alarm sounds, the two walks may take opposite steps if $X_i=x$, $X_i^{\BI}=y$ and $\xi_i$ falls between $q_x^{\BI}$ and $q_y^{\BI}$.  When (i) there is no alarm, (ii) $X_i=x, X_i^{\BI}=y$ for some $x,y$, and (iii) $\xi_i$ falls between $q_x^{\BI}$ and $q_y^{\BI}$, we say a {\it discrepancy} occurs at time $i$.  A {\it misstep} means either an alarm or a discrepancy.  For $h$ sufficiently large, for $x \geq h, y \geq h$, conditioned on $X_i=x, X_i^{\BI}=y$ and no alarm, the probability of a discrepancy is
\begin{equation} \label{discrep}
  |q_x^{\BI} - q_y^{\BI}| \leq \frac{\delta}{2h^2}|x-y|.
  \end{equation}
We let $N(k)$ denote the number of missteps which occur up to time $k$. 

Note that if $\delta=0$, the imbedded RW is symmetric and there are no discrepancies.  

When we couple $\{X_n\}$ and $\{X_n^{\BI}\}$ in the above manner, with both processes starting at $k$, we denote the corresponding measure by $P_k^*$.  Where confusion seems possible, for hitting times we then use a superscript to designate the process that the hitting time refers to, e.g. $\tau_0^{\Be}$ and $\tau_0^{\BI}$ for the Bessel process and its imbedded RW, respectively.


\section{Proof of the tail approximation \eqref{exctail}}

Recall that for \eqref{exctail} we have $\delta>-1$.
Let $\theta>0$, $0<\rho<1/8$, $0 < \ep_1 < \ep_2 < \sqrt{\rho}$ and $h_i = \ep_i \sqrt{m}$.  Let $0 < \eta < \ep_1/4$ and $h_{1\pm} = (\ep_1 \pm 2\eta)\sqrt{m}$.  
To prove \eqref{exctail} we will show that provided $\rho,\theta$ are sufficiently small, one can choose the other parameters so that
the following sequence of six inequalities holds, for large $m$:
\begin{align} \label{toprove}
  \frac{1-3\theta}{M_{h_2}} &P_{h_2}^{\Be}\big(\tau_0 \geq (1+2\rho)m\big) \\
  &\leq \frac{1-\theta}{M_{h_2}} P_{h_2}^{\BI}(\tau_{h_{1+}} \geq m) \notag \\
  &\leq \frac{1}{M_{h_2}} P_{h_2}(\tau_{h_1} \geq m) \notag \\
  &\leq P_0(\tau_0 \geq m) \notag \\
  &\leq \frac{1+\theta}{M_{h_2}} P_{h_2}(\tau_{h_1} \geq (1 - 2\rho)m) \notag \\
  &\leq \frac{1+2\theta}{M_{h_2}} P_{h_2}^{\BI}\big(\tau_{h_{1-}} \geq (1-2\rho)m\big) \notag \\
  &\leq \frac{1+4\theta}{M_{h_2}} P_{h_2}^{\Be}\big(\tau_0 \geq (1-3\rho)m\big). \notag
\end{align}
These may be viewed as three ``sandwich'' bounds on $P_0(\tau_0 \geq m)$, with the outermost sandwich readily yielding the desired result, as we will show.  The innermost sandwich (the 3rd and 4th inequalities) may be interpreted as follows.   For convenience we assume the $h_i$ are even integers. Recall $H$ from \eqref{Hdef}; when $H \geq h_2$, we let $T$ denote the first hitting time of $h_1$ after $\tau_{h_2}$.  We can decompose an excursion of height at least $h_2$ and length at least $m$ into 3 parts:  0 to $\tau_{h_2}$, $\tau_{h_2}$ to $T$, and $T$ to the end.  The idea is that for a typical excursion of length at least $m$, most of the length $\tau_0$ of the full excursion will be in the middle interval $[\tau_{h_2},T]$; the first and last intervals will have length at most $\rho m$.  The middle sandwich (2nd and 5th inequalities) comes from approximating the original RW by the imbedded RW from a Bessel process, during the interval $[\tau_{h_2},T]$.  Then the outermost sandwich (1st and 6th inequalities) comes from approximating the imbedded RW by the actual Bessel process, and from showing that the third interval, from $T$ to excursion end, is typically relatively short.

A useful inequality is as follows: for $h>k \geq 0$ and $m \geq 1$,  
\begin{align} \label{excdecomp2}
  P_0&(\tau_0 \geq m, H \geq h) \geq P_0\left( \tau_h < \tau_0 \right) P_{h}(\tau_k \geq m)  
    = \frac{1}{M_h} P_h(\tau_k \geq m).
  \end{align}
As a special case we have
\begin{align} \label{excdecomp}
  P_0(\tau_0 \geq m) \geq P_0&(\tau_0 \geq m, H \geq h_2) \geq \frac{1}{M_{h_2}} P_{h_2}(\tau_{h_1} \geq m),
  \end{align}
which establishes the 3rd inequality in \eqref{toprove}.

By \eqref{Mxapprox} there exists $l_1 \geq 1$ such that for all $x \geq l_1$,
\[
  x|R_x| \leq \half, \quad \frac{2\kappa M_x}{K_0x^{2\kappa}L(x)^{-1}}
    \in \left( \frac{7}{8},\frac{9}{8} \right), \quad \frac{2\kappa(M_{2x}-M_x)}{K_0(2^{2\kappa}-1)x^{2\kappa}L(x)^{-1}}
    \in \left( \frac{7}{8},\frac{9}{8} \right),
  \]
If $\delta\neq 0$, enlarging $l_1$ if necessary, we also have
\[
  \left| x(2p_x-1) + \frac{\delta}{2} \right| < \frac{|\delta|}{4}.
  \]

We turn to the 4th inequality in \eqref{toprove}.  We have
\begin{equation} \label{Hsplit}
  P_0(\tau_0 \geq m) = P_0(\tau_0 \geq m, H \geq h_2) + P_0(\tau_0 \geq m, H<h_2).
  \end{equation}
The main contribution should come from the first probability on the right.  To show this, we first need two lemmas.  
We begin with the following bound on strip-confinement probabilities.

\begin{lemma} \label{strip}
Assume \eqref{elliptic} and  \eqref{pxRx}.
There exists $K_{2}(\ep,l_1)$ as follows.  For all $h \geq 1, m \geq 2h^2$ and $0<q<h$,
\[
  P_q(X_n \in (0,h) \text{ for all } n \leq m) \leq e^{-K_{2} m/h^2}.
  \]
\end{lemma}
\begin{proof}
Consider first $\delta \neq 0,\ h > l_1$.  We claim that 
\[
  P_q(X_n  \in (l_1,h) \text{ for all } n \leq h^2 - l_1) 
  \]
is bounded away from 1 uniformly in $q,h$ with $l_1 \leq q < h$.  In fact, from the definition of $l_1$, the drift $p_x - q_x$ has constant sign for $x \geq l_1$.  Suppose the drift is positive; then $\{X_n\}$ and $\{X_n^{\sym}\}$ can be coupled so that $X_n \geq X_n^{\sym}$ for all $n$ up to the first exit time of $\{X_n\}$ from $(l_1,h)$.  Therefore
\[
  P_q(X_n  \in (l_1,h) \text{ for all } n \leq h^2 - l_1) \leq P_q^{\sym}(\tau_h > h^2 - l_1) \leq 1 - P_0^{\sym}(\tau_h \leq h^2 - l_1).
  \]
Since $X_n^{\sym}$ is a non-reflecting symmetric RW, for $Z$ a standard normal r.v. we have 
\[
  P_0^{\sym}(\tau_h \leq h^2 - l_1) \geq P_0^{\sym}(\tau_h \leq h^2/2) \geq 
    P_0^{\sym}(X_{\lfloor h^2/2 \rfloor}^{\sym} \geq h) \to P(Z > \sqrt{2})
  \]
as $h \to \infty$, so $P_0^{\sym}(\tau_h \leq h^2 - l_1)$ is bounded away from 0 uniformly in $h>l_1$, and the claim follows.  Similarly if the drift is negative, we can couple so that $X_n \leq X_n^{\sym}$ until the time that $\{X_n\}$ hits $l_1$, and therefore
\[
  P_q(X_n  \in (l_1,h) \text{ for all } n \leq h^2 - l_1) \leq P_q^{\sym}(\tau_{l_1} > h^2-l_1) \leq 1 - P_h^{\sym}(\tau_0 \leq h^2-l_1),
  \]
and the claim again follows straightforwardly.  Then since $q_x \geq \ep$ for all $x \leq l_1$, we have
\begin{align} \label{belowl1}
  P_q(X_n  \notin (0,h) \text{ for some } n \leq h^2) &\geq \ep^{l_1}P_q(X_n  \notin (l_1,h) \text{ for some } n \leq h^2 - l_1),
  \end{align}
which together with the claim shows that there exists $\gamma = \gamma(l_1,\ep)$ such that for all $l_1 \leq q < h$ we have
\begin{equation} \label{shortstrip}
  P_q(X_n  \notin (0,h) \text{ for some } n \leq h^2) \geq \gamma.
  \end{equation}
Therefore by straightforward induction, since $m \geq 2h^2$,
\begin{equation} \label{stripbound}
  P_q(X_n \in (0,h) \text{ for all } n \leq m) \leq (1-\gamma)^{\lfloor m/h^2 \rfloor} \leq e^{-K_{2} m/h^2},
  \end{equation}
completing the proof for $\delta \neq 0,\ h>l_1$.

For $\delta \neq 0, h \leq l_1$, the left side of \eqref{belowl1} is bounded below by $\ep^{l_1}$, and \eqref{stripbound} follows similarly.

For $\delta=0$, it seems simplest to proceed by comparison.  Instead, in place of \eqref{belowl1} we have
\begin{equation} \label{delta0}
  P_q(X_n  \notin (0,h) \text{ for some } n \leq h^2) \geq P_q(\tau_0 \leq q^2 + 1).
  \end{equation}
We can change the value of the (downward) drift parameter from $\delta=0$ to $\tilde{\delta} \in (-1,0)$ by subtracting $\tilde{\delta}/4x$ from $p_x$ for each $x \geq 1$.  By an obvious coupling, this reduces the probability on the right side of \eqref{delta0}.  But by Proposition \ref{intervalmid} below, this reduced probability is bounded away from 0 in $q \geq 1$.  Thus \eqref{shortstrip} and then \eqref{stripbound} hold in this case as well.
\end{proof}

It should be pointed out that the proof of Proposition \ref{intervalmid} makes use of Theorem \ref{tau0tail} which in turn makes use of Lemma \ref{strip}.  Since the application of Proposition \ref{intervalmid} in the proof of Lemma \ref{strip} is only for $\tilde{\delta} \neq 0$, and since this application is only used to prove the lemma in the case $\delta=0$, this is not circular---all proofs can be done for nonzero drift parameter first, and then this can be applied to obtain the result for 0 drift parameter.

If we start the RW at 0, we can strengthen the bound in Lemma \ref{strip}, as follows.  Let $Q_n = \max_{0 \leq k \leq n} X_k$, so $H = Q_{\tau_0}$. 

\begin{lemma} \label{lowheightlemma}
Assume $\delta>-1$.  There exist $K_{3}(\ep,l_1),K_{4}(\ep,l_1)$ as follows.  For all $h > l_1$ and $m \geq 4h^2$,
\[
  P_0(X_n \in (0,h) \text{ for all } 1 \leq n \leq m) \leq \frac{K_{3}}{M_h} e^{-K_{4} m/h^2}.
  \]
\end{lemma}
\begin{proof}
Let $k_1 = \min\{k: 2^{k-2} > l_1\}$ and $k_2 = \max\{k: 2^{k-1} < h\}$.  Then for some constants $K_i(\ep,l_1)$, 
{\allowdisplaybreaks
\begin{align} \label{excdecomp00a}
  P_0&(X_n \in (0,h) \text{ for all } 1 \leq n \leq m) \notag \\
  &\leq P_0\left( X_n \in (0,2^{k_1-1}) \text{ for all } 1 \leq n \leq m \right) 
    + \sum_{k=k_1}^{k_2} P_0\big(Q_m \in [2^{k-1},2^k), \tau_0 > m \big)  \notag \\
  &\leq e^{-K_{5}m} + 
    \sum_{k=k_1}^{k_2} \bigg[ P_0\big(Q_m \in [2^{k-1},2^k), \tau_0 > m, \tau_{2^{k-2}} \leq \frac{m}{2} \big) \notag \\
  &\qquad \qquad \qquad \qquad + P_0\big(Q_m \in [2^{k-1},2^k), \tau_0 > m, \tau_{2^{k-2}} > \frac{m}{2} \big) \bigg] \notag \\
  &\leq e^{-K_{5}m} + \sum_{k=k_1}^{k_2} \bigg[ P_0\left( \tau_{2^{k-2}} \leq \frac{m}{2}, \tau_0 > \tau_{2^{k-1}} \right) 
    P_{2^{k-2}}\left( X_n \in (0,2^k) \text{ for all } n \leq \frac{m}{2} \right) \notag \\
  &\qquad \qquad \qquad \qquad + P_0\left( \tau_0 > \tau_{2^{k-1}} > \tau_{2^{k-2}} > \frac{m}{2} \right) \bigg] \notag \\
  &\leq e^{-K_{5}m} + \sum_{k=k_1}^{k_2} \bigg[ P_1\left( \tau_0 > \tau_{2^{k-1}} \right) 
    e^{-K_{2} m/2^{2k+1}} \notag \\
  &\qquad \qquad \qquad \qquad + \frac{1}{ p_{2^{k-1}} \lambda_{2^{k-1}} } 
    P_{2^{k-1}}\left( \tau_0 < \tau_{2^{k-1}}, \tau_0 - \tau_{2^{k-2}} > \frac{m}{2} \right) \bigg] \notag \\
  &\leq e^{-K_{5}m} + \sum_{k=k_1}^{k_2} \bigg[ \frac{1}{M_{2^{k-1}}} e^{-K_{2} m/2^{2k+1}} \notag \\
  &\qquad \qquad +  \frac{1}{ p_{2^{k-1}} \lambda_{2^{k-1}} } 
    P_{2^{k-1}}\left( \tau_{2^{k-2}} < \tau_{2^{k-1}} \right) P_{2^{k-2}}\left( X_n \in (0,2^{k-1}) 
    \text{ for all } n \leq \frac{m}{2} \right) \bigg] \notag \\
  &\leq e^{-K_{5}m} + \sum_{k=k_1}^{k_2} \bigg[ \frac{1}{M_{2^{k-1}}} e^{-K_{2} m/2^{2k+1}} 
    +  \frac{1}{ p_{2^{k-1}} \lambda_{2^{k-1}} } 
    \frac{ q_{2^{k-1}}( M_{2^{k-1}} - M_{2^{k-1}-1} ) }{ M_{2^{k-1}} - M_{2^{k-2}} } e^{-K_{2} m/2^{2k-1}} \bigg] \\
  &\leq e^{-K_{5}m} + \sum_{k=k_1}^{k_2} \bigg[ \frac{1}{M_{2^{k-1}}} + \frac{1}{  M_{2^{k-1}} - M_{2^{k-2}} } \bigg] 
    e^{-K_{2} m/2^{2k+1}}\notag \\
  &\leq e^{-K_{5}m} + K_{6} \sum_{k=k_1}^{k_2} \frac{ L(2^k) }{ 2^{2k\kappa} } e^{-K_{2} m/2^{2k+1}} \notag \\
  &\leq e^{-K_{5}m} + K_{7} h^{-2\kappa} L(h) e^{-K_{2} m/8h^2} \notag \\
  &\leq K_{8} h^{-2\kappa} L(h) e^{-K_{9} m/h^2}, \notag
  \end{align} }
and the lemma follows from this and \eqref{Mxapprox}.  Here in the 2nd inequality we used the ellipticity condition \eqref{elliptic}, in the 4th inequality we used Lemma \ref{strip} and reversal of the path from time 0 to time $\tau_{2^{k-1}}$, in the 5th inequality we used \eqref{Hdef}, in the 6th inequality we used Lemma \ref{strip}, in the 8th inequality we used \eqref{lambda}, and in the last three inequalities we used the fact that $L$ is slowly varying.
\end{proof}

We return to the proof of the 4th inequality in \eqref{toprove}.  We have for $m$ sufficiently large that
\begin{align} \label{excdecomp0}
  P_0&(\tau_0 \geq m, H \geq h_2) \\
  &\leq P_0\left( \tau_{h_2} < \tau_0 \right) 
    P_{h_2}(\tau_{h_1} \geq (1 - 2\rho)m)  \notag \\
  &\qquad + P_0\left( \rho m < \tau_{h_2} < \tau_0 \right)
    + P_0\left( \tau_{h_2} < \tau_0 \right) P_{h_1}(\tau_0 > \rho m) \notag \\
  &\leq \frac{1}{M_{h_2}} P_{h_2}(\tau_{h_1} \geq (1 - 2\rho)m) + P_0\left( \rho m < \tau_{h_2} < \tau_0 \right) \notag \\
  &\qquad + \frac{1}{M_{h_2}} \left[ P_{h_1}(\tau_{h_2} < \tau_0) + P_{h_1}(\rho m < \tau_0 < \tau_{h_2}) \right] \notag \\
  &\leq \frac{1}{M_{h_2}} P_{h_2}(\tau_{h_1} \geq (1 - 2\rho)m) 
    + P_0\left( X_n \in (0,h_2) \text{ for all } 1 \leq n \leq \rho m \right) \notag \\
  &\qquad + \frac{1}{M_{h_2}} \left[ \frac{M_{h_1}}{M_{h_2}} 
    + P_{h_1}( X_n \in (0,h_2) \text{ for all } n \leq \rho m ) \right] \notag \\
  &\leq \frac{1}{M_{h_2}} P_{h_2}(\tau_{h_1} \geq (1 - 2\rho)m) + \frac{K_{3}}{M_{h_2}} e^{-K_{4}\rho/\ep_2^2} 
    + \frac{2}{M_{h_2}} \left( \frac{\ep_1}{\ep_2} \right)^{2\kappa} 
    + \frac{1}{M_{h_2}} e^{-K_{2}\rho/\ep_2^2} \notag \\
  &= (I) + (II) + (III) + (IV). \notag
 \end{align}
The 4th inequality in \eqref{excdecomp0} uses \eqref{lambda} and Lemmas \ref{strip} and \ref{lowheightlemma}.
We want to show that $(II), (III), (IV)$ are much smaller than $(I)$.  We will show that if $\ep_1 \ll \ep_2$ the probability in $(I)$ is of the same order as
\begin{equation} \label{longorder}
  P_{h_2}(\tau_{\sqrt{m}} < \tau_{h_1}) = \frac{M_{h_2} - M_{h_1}}{M_{\sqrt{m}} - M_{h_1}} \sim \ep_2^{2\kappa}.
  \end{equation}
This means that $(III) \ll (I)$ provided $\ep_1 \ll \ep_2^2$.
  
To complement \eqref{excdecomp0} we have the following bound from Lemma \ref{lowheightlemma}:
\begin{align} \label{excdecomp00}
  P_0&(\tau_0 \geq m, H < h_2) \leq P_0\left( X_n \in (0,h_2) \text{ for all } 1 \leq n < m \right) \leq 
    \frac{K_{3}}{M_{h_2}} e^{-K_{4}/\ep_2^2}.
  \end{align}
We will later prove the following lower bound for (I).
  
{\bf Claim 1.}   There exists $K_{10}(\delta)$ such that provided $\ep_1 < \ep_2/2$ and $m$ is sufficiently large, we have
\begin{equation} \label{claim1}
  P_{h_2}(\tau_{h_1} \geq (1 - 2\rho)m) \geq P_{h_2}\left(\tau_{h_1} \geq m \right) \geq K_{10}\ep_2^{2\kappa}
  \end{equation}
and
\begin{equation} \label{claim2}
  P_{h_2}^{\BI}(\tau_{h_1} \geq (1 - 2\rho)m) \geq P_{h_2}^{\BI}\left(\tau_{h_1} \geq m \right) \geq K_{10}\ep_2^{2\kappa}.
  \end{equation}
  
Assuming Claim 1, given $\theta>0$, provided $\ep_2$ and $\ep_1/\ep_2^2$ are sufficiently small (depending on $\delta,\rho,\theta$), the 4th inequality in \eqref{toprove} follows from \eqref{excdecomp0} and \eqref{excdecomp00}.
   
Our next task is to use the coupling of $\{X_n\}$ to $\{X_n^{\BI}\}$, from Section \ref{Coupling},
to prove the 2nd and 5th inequaltites in \eqref{toprove}.  Here $h_{1\pm}$ should be viewed as substitutes for $h_1$ which allow an error of $\eta\sqrt{m}$ in the coupling construction.  Fix $m/2 \leq l \leq m$.  We begin with the 5th inequality.  From the coupling construction we have
\begin{align} \label{deviation}
  P_{h_2}\big(\tau_{h_1} \geq l \big) &\leq P_{h_2}^{\BI}(\tau_{h_{1-}}^{\BI} \geq l \big) 
    + P_{h_2}^*( N(\tau_{h_{1-}}^{\BI}) \geq \eta \sqrt{m}, \tau_{h_{1-}}^{\BI} < l \wedge \tau_{h_1}).
  \end{align}
We need to bound the last probability.  Consider first $\delta \neq 0$.  Let $A(x) = \sup_{y \geq x} a(y)$, so $A(x) = o(1/x)$, and let $d_0 = h_{1-}^2A(h_{1-})/|\delta|$.  Suppose that for some time $i$ and some even integers $d_0 \leq d \leq \eta \sqrt{m}$, the gap $|X_i - X_i^{\BI}| \leq d$ and $X_i^{\BI} \geq h_{1-}$.  Provided $h_{1-}$ is large, by \eqref{discrep} the misstep probability for the next step is then at most 
\[
  A(h_{1-}) + \frac{|\delta| d}{h_{1-}^2} \leq \frac{2|\delta| d}{h_{1-}^2}.
  \]  
Let $G_{d_0}, G_{d_0+2}, \dots, G_{2\eta \sqrt{m}-2}$ be independent geometric random variables, with $G_d$ having parameter $2|\delta| d/h_{1-}^2$, and $S = G_{d_0} + G_{d_0+2} + \dots + G_{2\eta\sqrt{m}-2}$.  The gap $|X_i - X_i^{\BI}|$ can change (always by 2) only at times of missteps.  Therefore if we start from the time (if any) before $\tau_{h_{1-}}^{\BI}$ when the gap first reaches $d_0$, the time until the next misstep (if any) before $\tau_{h_{1-}}^{\BI}$ is stochastically larger than $G_{d_0}$, and then the time until the misstep after that (if any) before $\tau_{h_{1-}}^{\BI}$ is stochastically larger than $G_{d_0+2}$, and so on.  It follows that
\begin{equation} \label{Nbound}
  P_{h_2}^*( N(\tau_{h_{1-}}^{\BI}) \geq \eta \sqrt{m}, \tau_{h_{1-}}^{\BI} < l \wedge \tau_{h_1}) \leq P(S \leq l) \leq P(S \leq m).
  \end{equation}
Note that for $h_{1-}$ large (depending on $\eta/\ep_1$),
\[
  \frac{E(S)}{m} = \sum_{d_0\leq d < 2\eta\sqrt{m},\atop d-d_0\text{ even}} \frac{h_{1-}^2}{2|\delta| dm} 
    \geq \frac{h_{1-}^2}{4|\delta| m} \log \frac{\eta \sqrt{m}}{d_0}
    \geq \frac{\ep_1^2}{32|\delta|} \log \frac{|\delta|}{h_{1-}A(h_{1-})},
  \]
which grows to infinity as $m \to \infty$; thus $E(S) \gg m$.  In fact by standard computations using exponential moments, we obtain that for some $K_{11}(\eta,\delta,\ep_1)$ we have 
\begin{equation} \label{Sbound}
  P(S \leq m) \leq e^{-K_{11} \sqrt{m}}
  \end{equation}
for all sufficiently large $m$, and hence by Claim 1,
\begin{equation} \label{Sbound2}
  P(S \leq m) \leq \frac{\theta}{2} P_{h_2}(\tau_{h_1} \geq l).
  \end{equation}
In the case $\delta=0$, $\{X_n^{\BI}\}$ is a symmetric simple RW so there are no discrepancies, only alarms, which have probability at most $A(h_1)$ when the original RW is above height $h_1$.  Hence in place of \eqref{Nbound} we have the left side of \eqref{Nbound} bounded above by the probability that a Binomial($l,A(h_1)$) exceeds $\eta \sqrt{m}$, and this probability is also bounded by $e^{-K_{11} \sqrt{m}}$, and then the same argument applies.
Now \eqref{claim1}, \eqref{deviation}, \eqref{Nbound} and \eqref{Sbound2} show that provided $m$ is large, the 5th inequality in \eqref{toprove} holds.

Turning to the 2nd inequality in \eqref{toprove}, the analog of \eqref{Nbound} is still valid, so from the coupling construction,  \eqref{Sbound} and \eqref{claim2} (trivially modified to allow $h_{1+}$ in place of $h_1$), we have
\begin{align} \label{deviation2}
  P_{h_2}(\tau_{h_1} \geq m) &\geq P_{h_2}^{\BI}(\tau_{h_{1+}}^{\BI} \geq m) - 
     P_{h_2}^*( N(\tau_{h_1}) \geq \eta \sqrt{m}, \tau_{h_1} < m \wedge \tau_{h_{1+}}^{\BI}) \\
  &\geq P_{h_2}^{\BI}(\tau_{h_{1+}} \geq m) - e^{-K_{11}\sqrt{m}} \notag \\
  &\geq (1-\theta)P_{h_2}^{\BI}(\tau_{h_{1+}} \geq m), \notag
  \end{align}
proving the desired inequality.
  
The next step is to prove the first and last inequalities in \eqref{toprove}, by relating the probabilities for $\{X_n^{\BI}\}$ to probabilities for the continuous-time Bessel process $Y_t$.  We need to establish the following.

{\bf Claim 2.} Given $0 < \ep_1 < \ep_2, 0<\rho<1/3$ and $\theta>0$, for sufficiently large $m$,
\begin{equation} \label{BesselvsBI}
  P_{h_2}^{\BI}\big(\tau_{h_{1-}} \geq (1-2\rho)m\big) \leq (1+\theta) P_{h_2}^{\Be}\big(\tau_{h_{1-}} \geq (1-3\rho)m\big),
  \end{equation}
and 
\begin{equation} \label{BesselvsBI2}
  P_{h_2}^{\BI}(\tau_{h_{1+}} \geq m) \geq (1-\theta) P_{h_2}^{\Be}\big(\tau_{h_{1+}} \geq (1+\rho)m\big).
  \end{equation}

Suppose Claim 2 is proved.  For the Bessel process we have the obvious inequality
\begin{equation} \label{h1vs0}
  P_{h_2}^{\Be}\big(\tau_{h_{1-}} \geq (1-3\rho)m\big) \leq P_{h_2}^{\Be}\big(\tau_0 \geq (1-3\rho)m\big),
  \end{equation}
while
\begin{equation} \label{h2h10}
  P_{h_2}^{\Be}\big(\tau_{h_{1+}} \geq (1+\rho)m\big) \geq P_{h_2}^{\Be}\big(\tau_0 \geq (1+2\rho)m\big)
    - P_{h_{1+}}^{\Be}\big(\tau_0 \geq \rho m\big).
  \end{equation}
It follows from (15) in \cite{GJY03} that for $\delta>-1$ and $\ep>0$,
\begin{equation} \label{GammaRV}
  P_{\ep\sqrt{t}}^{\Be}(\tau_0 \geq t) = \int_0^{\ep^2/2} \frac{1}{\Gamma(\kappa)} u^{\kappa-1} e^{-u}\ du
    \sim K_{12} \ep^{2\kappa} \quad \text{as } \ep \to 0,
  \end{equation}
where $K_{12} = (2^\kappa \kappa \Gamma(\kappa))^{-1}$.
(Strictly speaking this seems to be stated in \cite{GJY03} only for Bessel processes with dimension in $(0,2)$, i.e. $\delta \in (-1,1)$, but the same proof works for nonpositive dimension, i.e. $\delta \geq 1$.  The key is the 3 lines after (57) in Appendix B of \cite{GJY03}.)  Applying this to each probability on the right side of \eqref{h2h10} we see that for $\rho$ and then $\ep_1/\ep_2$ taken sufficiently small and then $m$ large, we have 
\[
  P_{h_{1+}}^{\Be}\big(\tau_0 \geq \rho m\big) \leq \theta P_{h_2}^{\Be}\big(\tau_0 \geq (1+2\rho)m\big),
  \]
and therefore by \eqref{h2h10},
\begin{equation} \label{h1vs0part2}
  P_{h_2}^{\Be}\big(\tau_{h_{1+}} \geq (1+\rho)m\big) \geq (1-\theta) P_{h_2}^{\Be}\big(\tau_0 \geq (1+2\rho)m\big).
  \end{equation}
Combining \eqref{BesselvsBI2} and \eqref{h1vs0part2} we obtain the first inequality in \eqref{toprove}, while the last inequality in \eqref{toprove} is a consequence of \eqref{BesselvsBI} and \eqref{h1vs0}.  This completes the proof of \eqref{toprove}.  Since $\rho, \theta$ can be taken arbitrarily small, \eqref{toprove} together with \eqref{Mxapprox} and \eqref{GammaRV} proves \eqref{exctail}.

{\it Proof of Claim 2.}  Let $T_0=0$ and let $T_1,T_2,\dots$ be the stopping times when the Bessel process reaches an integer different from the last integer it has visited, so that $X_n^{\BI} = Y_{T_n}$.  Denote the hitting times of $h_{1-}$ in the two processes by $\tau_{h_{1-}}^{\BI}$ and $\tau_{h_{1-}}^{\Be}$ and let $\sigma_i = \min\{t: Y_t \in \{i-1,i+1\} \}$.  Given $k$ and $x_1,x_2,\dots,x_k$, with $x_i \geq h_{1-}$, let
\[
  A = \{\tau_{h_{1-}}^{\BI}=k\} \cap \{X_0^{\BI} = h_2, X_1^{\BI} = x_1,\dots,X_k^{\BI} = x_k \}.
  \]
Conditionally on $A$, the random variables $T_i - T_{i-1},\ i \leq k,$ are independent, with the distribution of $T_i - T_{i-1}$ being 
\[
  P_{x_{i-1}}^{\Be}\big( \sigma_{x_{i-1}} \in \cdot \mid Y_{\sigma_{x_{i-1}}} = x_i \big).
  \]
The mean of this distribution is 
\begin{equation} \label{mean}
  E_{h_2}^{\Be}(T_i - T_{i-1} \mid A) = \frac{ E_{x_{i-1}}^{\Be}\big( \sigma_{x_{i-1}} \delta_{\{ Y_{\sigma_{x_{i-1}}} = x_i \}} \big) }{ 
    P_{x_{i-1}}^{\Be}( Y_{\sigma_{x_{i-1}}} = x_i ) }.
  \end{equation}
We need estimates for the quantities
\[
  E_x^{\Be}(\sigma_x \delta_{\{Y_{\sigma_x} = x - 1\}}), \quad E_x^{\Be}(\sigma_x) \quad 
    \text{and} \quad P_x^{\Be}(Y_{\sigma_x} = x - 1).
  \]
Let 
\[
  s(x) = \begin{cases} x^{1+\delta} &\text{if } \delta \neq -1,\\ \log x &\text{if } \delta=-1 \end{cases}
  \]
be the scale function for the Bessel process and let $\mL f$ given by
\[
  (\mL f)(x) = \half f''(x) - \frac{\delta}{2x}f'(x)
  \]
be its infinitesmal generator.  For fixed $x$ and $z \in [x-1,x+1]$ the functions $f=f_x,g=g_x,h^\pm=h_x^\pm$ given by
\[
  f(z) = P_z^{\Be}(Y_{\sigma_x} = x - 1), \quad g(z) = E_z^{\Be}(\sigma_x), \quad 
    \frac{h^\pm(z)}{s(x+1)-s(x-1)} = E_z^{\Be}(\sigma_x \delta_{\{Y_{\sigma_x} = x \pm 1\}})
  \]
satisfy
\[
  \mL f \equiv 0, \quad f(x-1) = 1, \quad f(x+1) = 0; 
  \]
\[
  \mL g \equiv -1, \quad g(x-1) = g(x+1) = 0;
  \]
\[
  (\mL h^+)(z) = s(x-1) - s(z), \quad h^+(x-1) = h^+(x+1) = 0;
  \]
\[
  (\mL h^-)(z) = s(z) - s(x+1), \quad h^-(x-1) = h^-(x+1) = 0.
  \]
These can be solved explicitly, yielding that for $\delta>-1$,
\[
  f(z) = \frac{s(x+1) -s(z)}{s(x+1) - s(x-1)},
  \]
\[
  g(z) = \begin{cases} -\frac{1}{1-\delta} z^2 + \frac{4x}{1-\delta}\ \frac{1}{(x+1)^{1+\delta} - (x-1)^{1+\delta}} z^{1+\delta} + A_x
    &\text{if } \delta \neq 1,\\
    -z^2 \log z + \frac{(x+1)^2 \log (x+1) - (x-1)^2 \log (x-1)}{4x}z^2 +A_x'
    &\text{if } \delta=1, \end{cases}
  \]
\[
  h^+(z) = \begin{cases} \frac{(x-1)^{1+\delta}}{1-\delta} z^2 - \frac{1}{3+\delta} z^{3+\delta} + B_x z^{1+\delta} + D_x &\text{if } \delta \neq 1,\\
    (x-1)^2 z^2 \log z - \frac{1}{4} z^4 + B_x' z^2 + D_x' &\text{if } \delta=1, \end{cases}
  \]
\[
  h^-(z) = \begin{cases} -\frac{(x+1)^{1+\delta}}{1-\delta} z^2 + \frac{1}{3+\delta} z^{3+\delta} + B_x'' z^{1+\delta} + 
    D_x'' &\text{if } \delta \neq 1,\\
    -(x+1)^2 z^2 \log z + \frac{1}{4} z^4 + B_x''' z^2 + D_x''' &\text{if } \delta=1. \end{cases}
  \]
Note the formulas here for $\delta=1$ are determined by the formulas for $\delta\neq 1$, by continuity in $\delta$.
Here $B_x$ is given by
\[
  (1-\delta)B_x = -\frac{4x(x-1)^{1+\delta}}{(x+1)^{1+\delta} - (x-1)^{1-\delta}} + \frac{1-\delta}{1+\delta}(x-1)^2\psi_1\left( \frac{2}{x-1} \right)
  \]
with
\[
  \psi_1(u) = \frac{1+\delta}{3+\delta}\ \frac{(1+u)^{3+\delta}-1}{(1+u)^{1+\delta}-1} 
    = 1 + u + \frac{2+\delta}{6}u^2 + O(u^3) \quad \text{as } u \to 0,
  \]
$B_x'$ is given by
\[
  B_x' = \half (x^2+1) - \frac{(x-1)^2(x+1)^2}{4x}\log\left( 1 + \frac{2}{x-1} \right) - (x-1)^2\log(x-1),
  \]
and $B_x''$ and $B_x''$ are given by
\[
  (1-\delta)B_x'' = \frac{4x(x+1)^{1+\delta}}{(x+1)^{1+\delta} - (x-1)^{1-\delta}} -  \frac{1-\delta}{1+\delta}(x-1)^2 \psi_1\left( \frac{2}{x-1} \right)
  \]
and
\[
  B_x''' = -\half (x^2+1) + \frac{(x-1)^2(x+1)^2}{4x}\log\left( 1 + \frac{2}{x-1} \right) + (x+1)^2 \log(x+1) .
  \]
Finally, $A_x,A_x'$ and $D_x,D_x'$ and $D_x'',D_x'''$ are determined by $g(x-1)=0,h^+(x-1)=0$ and $h^-(x+1)=0$, respectively,
but we do not need these values because we can use for example $g(x) = g(x) - g(x-1)$, and $A_x$ or $A_x'$ cancels in the latter expression.
From these computations we readily obtain
\begin{equation} \label{limits}
  f(x) \to \half, \quad g(x) \to 1, \quad \frac{h^\pm(x)}{s(x+1)-s(x-1)}  \to \half \quad \text{as } x \to \infty,
  \end{equation}
and then also
\[
  E_x^{\Be}(\sigma_x \mid Y_{\sigma_x} = x-1) \to 1, \quad E_x^{\Be}(\sigma_x \mid Y_{\sigma_x} = x+1) \to 1
    \quad \text{as } x \to \infty.
  \]
Therefore, uniformly in those $A$ with all $x_i \geq h_{1-}$, as $m \to \infty$ we have 
\begin{equation} \label{near1}
  E_{h_2}^{\Be}(T_i - T_{i-1} \mid A) \to 1.
 \end{equation}
It is easily seen by comparison to ``Brownian motion plus small constant'' that $P_z^{\Be}(\sigma_x > 1)$ is bounded away from 1 uniformly in (large) $x$ and in $z \in [x-1,x+1]$.  Hence by the Markov property $P_x^{\Be}\big(\sigma_x > t)$ decays exponentially in $t$, uniformly in large $x$.  By \eqref{limits}, this means there exist $K_{13},K_{14}$ such that 
\[
  P_x^{\Be}\big(\sigma_x > t \mid Y_{\sigma_x} = x \pm 1 \big) \leq \max\left( \frac{1}{f(x)},\frac{1}{1-f(x)} \right) P_x^{\Be}\big(\sigma_x > t)
    \leq K_{13} e^{-K_{14} t},
  \]
for all $t\geq 0$ and all (large) $x$.  Therefore for $m$ sufficiently large, for all $A$ and $t$,
\begin{equation} \label{expbound}
  P_{h_2}^{\Be}(T_i - T_{i-1} > t \mid A) \leq K_{13} e^{-K_{14} t}.
  \end{equation}
By standard methods, it follows from \eqref{near1} and \eqref{expbound} that for some $K_{15}(\rho),K_{16}(\rho)$ not depending on $A$,
\begin{equation} \label{taugap}
  P_{h_2}^{\Be}\left( \left| \tau_{h_{1-}}^{\Be} - \tau_{h_{1-}}^{\BI} \right| > \rho  \tau_{h_{1-}}^{\BI}\ \big|\ A \right) = 
    P_{h_2}^{\Be}\left( \left| T_k - k \right| > \rho k\ \big|\ A \right) \leq K_{15} e^{-K_{16}k}.
  \end{equation}
Therefore the same bound holds unconditionally, so
\begin{align} \label{BesselvsBI3}
  P_{h_2}^{\BI}&\big(\tau_{h_{1-}}^{\BI} \geq (1-2\rho)m\big) \\
  &\leq P_{h_2}^{\Be}\big(\tau_{h_{1-}}^{\Be} \geq (1-3\rho)m\big) 
    + P_{h_2}^{\Be}\left( \tau_{h_{1-}}^{\BI} \geq (1-2\rho)m, \left| \tau_{h_{1-}}^{\Be} - \tau_{h_{1-}}^{\BI} \right| > \rho  \tau_{h_{1-}}^{\BI} \right)
    \notag \\
  &\leq P_{h_2}^{\Be}\big(\tau_{h_{1-}}^{\Be} \geq (1-3\rho)m\big) + K_{15} e^{-(1-2\rho)K_{16}m} \notag \\
  &\leq (1+\theta) P_{h_2}^{\Be}\big(\tau_{h_{1-}}^{\Be} \geq (1-3\rho)m\big), \notag
  \end{align}
where the last inequality follows from \eqref{h2h10} and \eqref{GammaRV}, for large $m$.  Thus \eqref{BesselvsBI} is proved.  We have similarly from \eqref{taugap} (with $\tau_{h_{1-}}$ trivially replaced by $\tau_{h_{1+}}$) that 
\begin{align} \label{BesselvsBI4}
  P_{h_2}^{\Be}&\big(\tau_{h_{1+}}^{\Be} \geq (1+\rho)m\big) \\
  &\leq P_{h_2}^{\Be}\big(\tau_{h_{1+}}^{\BI} \geq m\big)
    + P_{h_2}^{\Be}\big(\tau_{h_{1+}}^{\Be} \geq (1+\rho)m, \left| \tau_{h_{1+}}^{\Be} - \tau_{h_{1+}}^{\BI} \right| > \rho  \tau_{h_{1+}}^{\BI} \big)
    \notag \\
  &\leq P_{h_2}^{\BI}\big(\tau_{h_{1+}}^{\BI} \geq m\big) + K_{15}e^{-K_{16}m} \notag \\
  &\leq \frac{1}{1-\theta} P_{h_2}^{\BI}\big(\tau_{h_{1+}}^{\BI} \geq m\big), \notag
  \end{align}
so \eqref{BesselvsBI2}, and thus Claim 2, are also proved.

{\it Proof of Claim 1.}  From \eqref{BesselvsBI2}, \eqref{h1vs0part2} and then\eqref{GammaRV}, we have
\begin{align*}
  P_{h_2}^{\BI}(\tau_{h_{1+}} \geq m) &\geq (1-\theta) P_{h_2}^{\Be}\big(\tau_{h_{1+}} \geq (1+\rho)m\big) \\
  &\geq (1-2\theta) P_{h_2}^{\Be}\big(\tau_0 \geq (1+2\rho)m\big) \\
  &\geq K_{17} \ep_2^{2\kappa},
  \end{align*}
and it is straightforward to replace $\tau_{h_{1+}}$ here by $\tau_{h_1}$, proving the second inequality in \eqref{claim2}.  The first inequality there is trivial.

The first inequality in \eqref{claim1} is also trivial, so we prove the second one.  Using \eqref{claim2} and slight variants of \eqref{Nbound} and \eqref{Sbound} we get that for large $m$,
\begin{align*}
  P_{h_2}(\tau_{h_1} \geq m) &\geq P_{h_2}^*\big( \tau_{h_{1+}}^{\BI} \geq m, N(m) \leq \eta \sqrt{m} \big) \\
  &= P_{h_2}^*\big( \tau_{h_{1+}}^{\BI} \geq m \big) - P_{h_2}^*\big( \tau_{h_{1+}} \geq m, N(m) > \eta \sqrt{m} \big) \\
  &\geq K_{10} \ep_2^{2\kappa} - e^{-K_{2} \sqrt{m}} \\
  &\geq \half K_{10} \ep_2^{2\kappa},
  \end{align*}
completing the proof of Claim 1.

This also completes the proof of \eqref{exctail}, as noted after Claim 2.\\[10 pt]

\section{Proof of \eqref{excpoint} and \eqref{exctailsv}}

For even numbers $0<m<n$, let
\[
  f_m = P_0(\tau_0 = m), \qquad A_{m,n} = \frac{2}{n-m+2} \sum_{j=m}^n f_j.
  \]
$A_{m,n}$ is the average of the even-index $f_j$'s with $j \in [m,n]$.  We use \eqref{exctail2} and the following convexity property of $\{f_m\}$. 

\begin{lemma} \label{convexity}
For all even numbers $0<k<m$,
\begin{equation} \label{convex}
  f_m \leq\frac{f_{m+k} + f_{m-k}}{2}
  \end{equation}
and
\begin{equation} \label{average}
  f_m \leq A_{m-k,m+k}.
  \end{equation}
\end{lemma}
\begin{proof}
Let $\bx = \{x_0,\dots,x_m\}$ be the trajectory of an excursion of length $m$ starting at time 0, and $\bx' = \{x_k',\dots,x_{m+k}'\}$ the trajectory of an excursion of length $m$ starting at time $k$.  (Necessarily, then, $x_0 = x_m = x_k' = x_{m+k}' = 0$ and all other $x_j$ and $x_j'$ are positive.)  Since $k$ is even, there must be an $s \in (k,m)$ with $x_s = x_s'$; let $T = T(\bx,\bx')$ denote the least such $s$ and $D_t = \{(\bx,\bx'): T(\bx,\bx') = t\}$.  For $\bx,\bx' \in D_t$, by switching the two trajectories after time $t$, we obtain an excursion $\by = \by(\bx,\bx') = \{x_0,\dots,x_t,x_{t+1}',\dots,x_{m+k}'\}$ of length $m+k$ and an excursion $\by' = \by'(\bx,\bx') = \{x_k',\dots,x_t',x_{t+1},\dots,x_m\}$ of length $m-k$.  The map $(\bx,\bx') \mapsto (\by,\by')$ is one to one and satisfies
\[
  P_0(\bx) P(\bx' \mid X_k = 0) = P_0(\by) P(\by' \mid X_k = 0).
  \]
It follows that
\begin{align*}
  f_m^2 &= \sum_{t:k<t<m}\ \sum_{(\bx,\bx') \in D_t} P_0(\bx) P(\bx' \mid X_k = 0) \\
  &= \sum_{t:k<t<m}\ \sum_{(\bx,\bx') \in D_t} P_0(\by(\bx,\bx')) P(\by'(\bx,\bx') \mid X_k = 0) \\
  &\leq \left( \sum_{\by} P_0(\by) \right) \left( \sum_{\by'} P(\by' \mid X_k=0) \right) \\
  &= f_{m+k} f_{m-k} \\
  &\leq \left( \frac{f_{m+k} + f_{m-k}}{2} \right)^2.
  \end{align*}
Equation \eqref{average} is an immediate consequence of \eqref{convex}.
\end{proof}

Let $\theta>0$.  Provided $\eta$ is sufficiently small (depending on $\theta$), we have from \eqref{exctail2}, \eqref{psidef} and Lemma \ref{convexity} that for $n$ large and even and  $k = 2\lfloor \eta n/2 \rfloor$,
\begin{align} \label{pointupper}
  P_0(\tau_0 = n) &= f_n \\
  &\leq A_{n-k,n+k} \notag \\
  &= \frac{1}{k+1} P_0\left( (1-\eta)n \leq \tau_0 \leq (1+\eta)n \right) \notag \\
  &\leq (1+\theta) \frac{2^{2-\kappa}\kappa}{K_0 \Gamma(\kappa)} n^{-(\kappa+1)}L(\sqrt{n}). \notag
  \end{align}
In the reverse direction, suppose $f_n < (1-\theta) A_{n-k,n-2}$ for some $0<k<n/2$, with $k,n$ even.  By Lemma \ref{convexity} we have
\begin{align}
  \sum_{j=1}^{k/2} f_{n-2j} &\leq \frac{k}{4} f_n + \half \sum_{j=1}^{k/2} f_{n-4j} \\
  &\leq \frac{1-\theta}{2} \sum_{j=1}^{k/2} f_{n-2j} + \frac{1}{4} \sum_{j=1}^{k/2} f_{n-4j} 
    + \frac{1}{4} \sum_{j=1}^{k/2} \frac{f_{n-4j+2} + f_{n-4j-2}}{2} \notag \\
  &= \frac{1-\theta}{2} \sum_{j=1}^{k/2} f_{n-2j} + \frac{1}{4} \sum_{j=1}^{k/2} f_{n-4j} 
    + \frac{1}{8} \sum_{j=1}^{k/2} f_{n-4j+2} + \frac{1}{8} \sum_{j=2}^{(k+2)/2} f_{n-4j+2} \notag \\
  &= \frac{1-\theta}{2} \sum_{j=1}^{k/2} f_{n-2j} + \frac{1}{4} \sum_{m=2}^k f_{n-2m} + \frac{1}{8} f_{n-2} 
    + \frac{1}{8} f_{n-2k-2}, \notag
  \end{align}
and therefore
\begin{align}
  \frac{1+\theta}{2} \sum_{j=1}^{k/2} f_{n-2j} \leq \frac{1}{4} \sum_{j=1}^k f_{n-2j} + \frac{1}{8} f_{n-2k-2},
  \end{align}
which in the case $k = 2\lfloor \eta n/2 \rfloor$ gives
\begin{align} \label{contra}
  (1+\theta)&P_0\left( (1-\eta)n \leq \tau_0 \leq n-2 \right) \\
  &\leq \half P_0\left( (1 - 2\eta)n \leq \tau_0 \leq n-2 \right)
    + \frac{1}{4} P_0\left(\tau_0 = n - 4\left\lfloor \frac{\eta n}{2} \right\rfloor - 2 \right). \notag
\end{align}
For small $\eta$ and large $n$, this contradicts \eqref{exctail}, showing that we cannot have $f_n < (1-\theta) A_{n-k,n-2}$.  Therefore for large $n$, using \eqref{exctail} we have
\begin{align}
  P_0(\tau_0 = n) &= f_n \\
  &\geq (1-\theta)A_{n-k,n-2} \notag \\
  &= \frac{2(1-\theta)}{k} P_0\left( (1 - \eta)n \leq \tau_0 \leq n-2 \right) \notag \\
  &\geq (1-2\theta) \frac{2^{2-\kappa}\kappa}{K_0 \Gamma(\kappa)} n^{-(\kappa+1)}L(\sqrt{n}). \notag
\end{align}
This and \eqref{pointupper} prove \eqref{excpoint}.

We now prove \eqref{exctailsv}. Let $\tP$ denote the distribution of the Bessel-like RW dual to $P$, that is, the walk with transition probabilities $\tilde{p}_x = q_x, \tilde{q}_x = p_x$ for $x \geq 1$.  In \cite{DFPS97} it is proved that for $n$ even,
\begin{equation} \label{duality}
  P(\tau_0 > n) = \tP(X_n = 0).
  \end{equation}
For $\delta=-1$, the dual walk has drift parameter $\tilde{\delta}=1$, so \eqref{exctailsv} follows by applying \eqref{excpoint} and \eqref{c2case3} to the dual walk.

\section{Proof of Theorem \ref{hittime}}

We want to use \eqref{reverse} so we need to approximate 
\[
  f_n^{(k)} = P_k(\tau_0 = n) \quad \text{ and} \quad P_k(X_n=0).
  \]
We sometimes omit the superscript $(k)$ when it is equal to 0.  We start with the following relative of Lemma \ref{convexity}.  

\begin{lemma} \label{latticepath}
Let $0 \leq p \leq q \leq r \leq s \leq \infty$ and $0 \leq k < l$.  Then for $l-k$ even,
\begin{equation} \label{intervals}
  P_l(\tau_0 \in [p,q]) P_k(\tau_0 \in [r,s]) \leq P_l(\tau_0 \in [r,s]) P_k(\tau_0 \in [p,q]),
\end{equation}
and for $l-k$ odd,
\begin{equation} \label{intervals2}
  P_l(\tau_0 \in [p,q]) P_k(\tau_0 \in [r,s]) \leq P_l(\tau_0 \in [r+1,s+1]) P_k(\tau_0 \in [p-1,q-1]).
\end{equation}
\end{lemma}
\begin{proof}
Suppose first that $l-k$ is even.  Consider a lattice path $\bx$ starting at $(0,k)$ in space-time which first hits the horizontal axis at a time in $[r,s]$, and a lattice path $\bx'$ starting at $(0,l)$ which first hits the axis at a time in $[p,q]$.  Since $l-k$ is even, there must be a $t \in (0,q]$ with $x_t = x_t'$.  Switching the two trajectories after the first such $t$ and proceeding as in Lemma \ref{convexity} we obtain \eqref{intervals}.

For $l-k$ odd we repeat this argument but with the path $\bx$ shifted one unit to the right, that is, started from $(1,k)$.
\end{proof}

Here are some special cases of interest for Lemma \ref{latticepath}, particularly when comparing point versus interval probabilities for $\tau_0$.

\begin{corollary} \label{latticepath2}
(i) For all $0 \leq k<l \leq n$ and $j>0$ with $n-l$ and $n+j-k$ even,
\begin{equation} \label{lattice1}
  \frac{f_{n+j}^{(k)}}{f_n^{(k)}} \leq \frac{f_{n+j}^{(l)}}{f_n^{(l)}} \quad \text{if $l-k$ is even},
  \end{equation}
and
\begin{equation} \label{lattice2}
  \frac{f_{n+j}^{(k)}}{f_{n-1}^{(k)}} \leq \frac{f_{n+j+1}^{(l)}}{f_n^{(l)}} \quad \text{if $l-k$ is odd}.
  \end{equation}
(ii) For all $0 \leq l \leq m$,
\begin{equation} \label{lattice3}
  P_l(\tau_0 = m) \leq M_l P_0(\tau_0 = m) \quad \text{if $l$ is even},
\end{equation}
and
\begin{equation} \label{lattice4}
  P_l(\tau_0 = m) \leq M_l P_0(\tau_0 = m-1) \quad \text{if $l$ is odd}.
\end{equation}
(iii) For all $l>0$ and $0 \leq p < q < m$,
\begin{equation} \label{lattice5}
  P_l(\tau_0 = m) \geq M_l P_0(\tau_0 = m) \frac{ P_0(\tau_0 - \tau_l \in [p,q]) }{ P_0(\tau_0 \in [p,q]) } \quad \text{if $l$ is even},
\end{equation}
and
\begin{equation} \label{lattice5a}
  P_l(\tau_0 = m) \geq M_l P_0(\tau_0 = m-1) \frac{ P_0(\tau_0 - \tau_l \in [p,q]) }{ P_0(\tau_0 \in [p-1,q-1]) } \quad \text{if $l$ is odd}.
\end{equation}

\end{corollary}

By Corollary \ref{latticepath2}, to show that $P_l(\tau_0 = m)$ can be well approximated by $M_l P_0(\tau_0 = m)$ (or $M_l P_0(\tau_0 = m-1)$, depending on parity), it is sufficient to find, given $m$, values $p<q \leq m$ for which the fraction in \eqref{lattice5} or \eqref{lattice5a} is almost 1.  We will see that this can be done for $m \gg l^2$.

\begin{proof}[Proof of Corollary \ref{latticepath2}]
(i) Take $p=q=n$ and $r=s=n+j$ in Lemma \ref{latticepath} to get $f_n^{(l)} f_{n+j}^{(k)} \leq f_{n+j}^{(l)} f_n^{(k)}$ in the case of even $l-k$, and similarly for odd $l-k$.

(ii) Consider even $l$.  We may assume $m$ is also even, for otherwise the left side of \eqref{lattice3} is 0.  Applying Lemma \ref{latticepath} with $k=0$, $p=q=r=m$ and $s=\infty$ we get
\begin{equation} \label{lattice6}
  P_l(\tau_0 = m) \leq \frac{ P_l(\tau_0 \geq m) }{ P_0(\tau_0 \geq m) } P_0(\tau_0 = m),
\end{equation}
while by \eqref{excdecomp2},
\begin{align}
  \frac{1}{M_l} P_l(\tau_0 \geq m) \leq P_0(\tau_0 \geq m).
  \end{align}
Together these prove \eqref{lattice3}.  For odd $l$ we may assume $m$ is odd, and in place of \eqref{lattice6} we get 
\begin{equation} \label{lattice7}
  P_l(\tau_0 = m) \leq \frac{ P_l(\tau_0 \geq m+1) }{ P_0(\tau_0 \geq m) } P_0(\tau_0 = m-1),
\end{equation}
and the rest of the proof is essentially unchanged, since $P_l(\tau_0 \geq m+1) \leq P_l(\tau_0 \geq m)$.

(iii) Consider even $l$.  We may assume $m$ is even, for otherwise the right side of \eqref{lattice5} is 0.  Applying Lemma \ref{latticepath} with $k=0$, $[r,s] = \{m\}$ we obtain
\begin{equation} \label{lattice8}
  P_l(\tau_0 = m) \geq \frac{ P_l(\tau_0 \in [p,q]) }{ P_0(\tau_0 \in [p,q]) } P_0(\tau_0 = m),
\end{equation}
while by \eqref{Hdef},
\begin{equation} \label{1overMl}
  \frac{1}{M_l} P_l(\tau_0 \in [p,q]) = P_0(\tau_l < \tau_0) P_l(\tau_0 \in [p,q]) = P_0(\tau_0 - \tau_l \in [p,q]),
\end{equation}
and together these prove \eqref{lattice5}.  For odd $l$ we may again assume $m$ is odd and take $[r,s] = \{m-1\}$, so that in place of \eqref{lattice8}, using \eqref{1overMl}  we get 
\begin{align} \label{lattice9}
  P_l(\tau_0 = m) &\geq \frac{ P_l(\tau_0 \in [p,q]) }{ P_0(\tau_0 \in [p-1,q-1]) } P_0(\tau_0 = m-1) \\
  &= M_l  \frac{ P_0(\tau_0 - \tau_l \in [p,q]) }{ P_0(\tau_0 \in [p-1,q-1]) } P_0(\tau_0 = m-1). \notag
\end{align}
\end{proof}

Note that if we take $k=0$ and $j \ll n$ in \eqref{lattice1}, we see from \eqref{excpoint} that the left side of \eqref{lattice1} is close to 1, so the right side cannot be much less than 1 for any $l>0$.

For the Bessel process we have by \eqref{GammaRV} that for $0<a<b$, recalling $\kappa = (1+\delta)/2$,
\begin{align} \label{Besselcase}
  P_k^{\Be}(\tau_0 \in [a,b]) &= \int_{k^2/2b}^{k^2/2a} \frac{1}{\Gamma(\kappa)}u^{\kappa-1} e^{-u}\ du.
  \end{align}
As a step toward approximating $P_k(X_n=0)$ we have the following ``interval'' version of Theorem \ref{hittime}, for midrange starting heights ($k$ of order $\sqrt{m}$); for these we apparently cannot get sharp results from Corollary \ref{latticepath2}(ii) and (iii).

\begin{proposition} \label{intervalmid}
Let $\theta>0,\chi>0$, $0 < \Dm < \DM$ and $0<a<b$.  Provided $\chi$ is sufficiently small (depending on $\theta$), $\DM$ is sufficiently small (depending on $\theta,\chi$), 
\begin{equation} \label{ab}
  \frac{b-a}{b} \in [\Delta_{\min},\Delta_{\max}],
  \end{equation}
the starting height $k$ is midrange, that is,
\begin{equation} \label{mk2}
  \sqrt{a\chi} \leq k \leq \sqrt{a/\chi},
  \end{equation}
and $a$ is sufficiently large (depending on $\theta,\chi,\Dm,\DM$), we have
\begin{align} \label{sandwich7}
  (1-\theta) &P_k^{\Be}(\tau_0 \in [a,b]) \leq  P_k(\tau_0 \in [a,b]) 
  \leq (1+\theta) P_k^{\Be}(\tau_0 \in [a,b]) 
  \end{align}
and
\begin{align} \label{sandwich6}
  \frac{1-\theta}{\Gamma(\kappa)} \frac{b-a}{b} \left( \frac{k^2}{2a} \right)^\kappa e^{-k^2/2a}
    &\leq P_k(\tau_0 \in [a,b]) \\
  &\leq \frac{1+\theta}{\Gamma(\kappa)} \frac{b-a}{b} \left( \frac{k^2}{2a} \right)^\kappa e^{-k^2/2a}. \notag
  \end{align}
  \end{proposition}
  
\begin{proof}
Let $0<\rho<\Delta_{\min}/8$ and $\zeta>2\beta>0$.  We always select our constants in the following manner:  $\theta$ is given; we choose $\chi$ then $\Delta_{\max}$, and then $\Delta_{\min} < \Delta_{\max}$ is arbitrary, then we choose $\rho$ and then $\zeta$ and $\beta$ (which appear in \eqref{toprove2} below.)  Finally we choose $[a,b]$ as specified.  Each choice may depend only on the preceding choices, and when we say a parameter is ``sufficiently large'' (or small), the required size may depend on the previous choices.

The general outline is similar to the proof of \eqref{exctail}.  Analogously to \eqref{toprove}, we will establish the following sequence of ten inequalities:
\begin{align} \label{toprove2}
  (1-&6\theta)P_k^{\Be}(\tau_0^{\Be} \in [a,b]) \\
  &\leq (1-5\theta)P_k^{\Be}\big(\tau_0^{\Be} \in [(1+2\rho)a,(1-3\rho)b] \big) \notag \\
  &\leq (1-4\theta)P_k^{\Be}\big(\tau_{\zeta k}^{\Be} \in [(1+\rho)a,(1-3\rho)b] \big) \notag \\
  &\leq (1-3\theta) P_k^{\BI}\big( \tau_{\zeta k}^{\BI} \in [a,(1-2\rho)b] \big) \notag\\
  &\leq (1-\theta) P_k\big( \tau_{(\zeta-2\beta)k} \in [a,(1-\rho)b] \big) \notag \\
  &\leq P_k(\tau_0 \in [a,b]) \notag \\
  &\leq (1+\theta) P_k\big( \tau_{(\zeta+2\beta)k} \in [(1-\rho)a,b] \big) \notag \\
  &\leq (1+4\theta)P_k^{\BI}\big( \tau_{\zeta k}^{\BI} \in [(1-\rho)a,(1+\rho)b] \big) \notag \\
  &\leq (1+5\theta)P_k^{\Be}(\tau_{\zeta k}^{\Be} \in [(1-2\rho)a,(1+2\rho)b]) \notag \\
  &\leq (1+7\theta)P_k^{\Be}(\tau_0^{\Be} \in [(1-2\rho)a,(1+3\rho)b]) \notag \\
  &\leq (1+8\theta)P_k^{\Be}(\tau_0^{\Be} \in [a,b]). \notag
  \end{align}
As with \eqref{toprove}, this should be viewed as five ``sandwich'' bounds on $P_k(\tau_0 \in [a,b])$, with the outermost sandwich yielding the desired result.

Provided $\Delta_{\max}/\theta\chi$ is sufficiently small and the second inequality in \eqref{mk2} holds, the gamma density
\[
  f_\kappa(u) =  \frac{1}{\Gamma(\kappa)}u^{\kappa-1} e^{-u}, \quad u \geq 0,
  \]
satisfies
\[
  (1-\theta) f_\kappa\left( \frac{k^2}{2a} \right) \leq f_\kappa(u) \leq (1+\theta) f_\kappa\left( \frac{k^2}{2a} \right) \quad \text{for all } u \in    
    \left[ \frac{k^2}{2b},\frac{k^2}{2a} \right].
  \]
Then by \eqref{Besselcase}, 
\begin{equation} \label{Bessellower}
  P_k^{\Be}(\tau_0 \in [a,b]) \geq \frac{1-\theta}{\Gamma(\kappa)} \frac{b-a}{b} \left( \frac{k^2}{2a} \right)^\kappa e^{-k^2/2a} 
    \end{equation}
and, using also the second inequality in \eqref{mk2},
\begin{equation} \label{Besselupper}
  P_k^{\Be}(\tau_0 \in [a,b]) \leq \frac{1+\theta}{\Gamma(\kappa)} \frac{b-a}{b} \left( \frac{k^2}{2a} \right)^\kappa e^{-k^2/2a} .
  \end{equation}
Therefore \eqref{sandwich6} follows from \eqref{sandwich7}.  The inequalities \eqref{Bessellower} and \eqref{Besselupper}, with minor modifications made to $\theta,a$ and $b$, also prove the first and last inequalities in \eqref{toprove2}, provided $\rho$ is suficiently small (depending on $\theta,\Dm,\chi$.)  

Turning to the 2nd and 9th inequalities in \eqref{toprove2}, provided $\zeta^2/\rho \chi$ is sufficiently small (depending on $\theta$), using \eqref{Besselcase} we have
\begin{align} \label{zetakvs0}
  P_k^{\Be}(\tau_0 \in [(1-2\rho)a,(1+3\rho)b]) &\geq P_k^{\Be}\big(\tau_{\zeta k} \in [(1-2\rho)a,(1+2\rho)b] \big) 
    P_{\zeta k}^{\Be}\big( \tau_0 \leq \rho b \big) \\
  &= P_k^{\Be}\big(\tau_{\zeta k} \in [(1-2\rho)a,(1+2\rho)b] \big) \int_{\zeta^2k^2/2\rho b}^\infty 
    \frac{1}{\Gamma(\kappa)}u^{\kappa-1} e^{-u}\ du \notag\\
  &\geq (1-\theta) P_k^{\Be}\big(\tau_{\zeta k} \in [(1-2\rho)a,(1+2\rho)b] \big) . \notag
  \end{align}
This proves the 9th inequality in \eqref{toprove2}.
In the other direction, 
\begin{align} \label{zetakvs0b}
  P_k^{\Be}(\tau_0 \in [(1+2\rho)a,(1-3\rho)b]) &\leq P_k^{\Be}\big(\tau_{\zeta k} \in [(1+\rho)a,(1-3\rho)b] \big) 
    + P_{\zeta k}^{\Be}\big( \tau_0 > \rho a \big).
  \end{align}
From \eqref{ab}, \eqref{mk2} and \eqref{Bessellower}, 
\begin{equation} \label{Bessellower2}
  P_k^{\Be}(\tau_0 \in [(1+2\rho)a,(1-3\rho)b]) 
    \geq \frac{(1-\theta)\Dm}{\Gamma(\kappa)} \left( \frac{k^2}{2a} \right)^\kappa e^{-1/2\chi},
    \end{equation}
and hence by \eqref{GammaRV}, provided $\zeta^2/\rho$ is sufficiently small (depending on $\theta,\Dm,\chi$),
\begin{align} \label{Besselupper2}
  P_{\zeta k}^{\Be}\big( \tau_0 > \rho a \big) &= \int_0^{\zeta^2 k^2/2\rho a} \frac{1}{\Gamma(\kappa)}u^{\kappa-1} e^{-u}\ du \\
  &\leq \frac{1}{\kappa \Gamma(\kappa)}  \left( \frac{\zeta^2 k^2}{2\rho a} \right)^\kappa \notag \\
  &\leq \theta P_k^{\Be}(\tau_0 \in [(1+2\rho)a,(1-3\rho)b]). \notag
  \end{align}
With \eqref{zetakvs0b} this shows that
\begin{equation} \label{zetakvs0c}
  P_k^{\Be}(\tau_0 \in [(1+2\rho)a,(1-3\rho)b]) \leq \frac{1}{1-\theta} P_k^{\Be}\big(\tau_{\zeta k} \in [(1+\rho)a,(1-3\rho)b] \big),
  \end{equation}
which proves the 2nd inequality in \eqref{toprove2}.

Next we consider the 3rd and 8th inequalities in \eqref{toprove2}, in which Bessel-process probabilities are compared to similar probabilities for the imbedded RW.  First, for the 8th inequality, analogously to \eqref{BesselvsBI3} we have for some $K_i = K_i(\rho)$ that
\begin{align} \label{BIvsBe1}
  P_k^{\BI}&\big(\tau_{\zeta k}^{\BI} \in [(1-\rho)a,(1+\rho)b] \big) \\
  &\leq P_k^{\Be}\big(\tau_{\zeta k}^{\Be} \in [(1-2\rho)a,(1+2\rho)b] \big) \notag \\
  &\qquad + P_k^{\Be}\big( \tau_{\zeta k}^{\BI} \in [(1-\rho)a,(1+\rho)b],
    |\tau_{\zeta k}^{\BI} - \tau_{\zeta k}^{\Be}| > \rho a \big) \notag \\
  &\leq P_k^{\Be}\big(\tau_{\zeta k}^{\Be} \in [(1-2\rho)a,(1+2\rho)b] \big) + K_{18}e^{-K_{19}a}. \notag
  \end{align}
By \eqref{mk2}, \eqref{zetakvs0c} and \eqref{Bessellower2}, there exist $K_{20} = K_{20}(\rho,\Dm,\chi)$ and $K_{21} = K_{21}(\rho,\Dm,\chi,\theta)$ 
such that for $a \geq K_{21}$,
\begin{equation} \label{expsmall}
  P_k^{\Be}\big(\tau_{\zeta k}^{\Be} \in [(1-2\rho)a,(1+2\rho)b] \big) \geq K_{20} a^{-\kappa} \geq \frac{1}{\theta}e^{-K_{19}a},
  \end{equation}
which with \eqref{BIvsBe1} and \eqref{zetakvs0} shows that 
\begin{align} \label{BIvsBe2}
  P_k^{\BI}\big(\tau_{\zeta k}^{\BI} \in [(1-\rho)a,(1+\rho)b] \big) &\leq (1+\theta)P_k^{\Be}\big(\tau_{\zeta k}^{\Be} \in [(1-2\rho)a,(1+2\rho)b] \big),
  \end{align}
so the 8th inequality in \eqref{toprove2} is proved.
For the 3rd inequality, similarly to \eqref{BIvsBe1} and \eqref{BIvsBe2} we get
\begin{equation}\label{BIvsBe3}
  P_k^{\Be}\big(\tau_{\zeta k}^{\Be} \in [(1+\rho)a,(1-3\rho)b] \big) \leq P_k^{\BI}\big( \tau_{\zeta k}^{\BI} \in [a,(1-2\rho)b] \big)
    + e^{-K_{19}a},
    \end{equation}
which together with a slight modification of \eqref{expsmall} gives
\begin{equation} \label{BIvsBe4}
  (1-\theta)P_k^{\Be}\big(\tau_{\zeta k}^{\Be} \in [(1+\rho)a,(1-3\rho)b] \big) \leq P_k^{\BI}\big( \tau_{\zeta k}^{\BI} \in [a,(1-2\rho)b] \big),
  \end{equation}
yielding the desired result.
  
Now we consider the 4th through 7th inequalities in \eqref{toprove2}, comparing probabilities for the imbedded RW to similar probabilities for the original RW, and comparing the hitting times of $(\zeta-2\beta)k$ and 0; for this we use the coupling of $\{X_n\}$ and $\{X_n^{\BI}\}$.  First, for the 4th inequality, observe that for walks starting at $k$, if $\tau_{\zeta k}^{\BI} \in [a,(1-2\rho)b]$ and the number of missteps by time $\tau_{\zeta k}^{\BI}$ is less than $\beta k$, then at time $\tau_{\zeta k}^{\BI}$, the stopping time $\tau_{(\zeta-2\beta)k}$ for the RW $\{X_n\}$ has not yet occurred and this RW is located in $((\zeta-2\beta)k,(\zeta+2\beta)k)$.  Therefore
\begin{align} \label{BIvsRW1}
  P_k^{\BI}&\big( \tau_{\zeta k}^{\BI} \in [a,(1-2\rho)b] \big) \\
  &\leq P_k^*\big( \tau_{(\zeta-2\beta)k} < \tau_{\zeta k}^{\BI} \wedge (1-2\rho)b, N(\tau_{(\zeta-2\beta)k}) \geq \beta k \big) \notag \\
  &\quad + P_k^*\big( \tau_{\zeta k}^{\BI} \in [a,(1-2\rho)b], \tau_{(\zeta-2\beta)k} > \tau_{\zeta k}^{\BI}, 
  X_{\tau_{\zeta k}^{\BI}} \in ((\zeta-2\beta)k,(\zeta+2\beta)k) \big). \notag
\end{align}
Let 
\[
  D = \{ \tau_{\zeta k}^{\BI} \in [a,(1-2\rho)b], \tau_{(\zeta-2\beta)k} > \tau_{\zeta k}^{\BI}, 
    X_{\tau_{\zeta k}^{\BI}} \in ((\zeta-2\beta)k,(\zeta+2\beta)k) \}
  \]
denote the last event in \eqref{BIvsRW1}.  When $D$ occurs, the RW $\{X_n^{\BI}\}$ reaches height $\zeta k$ at some time $l$, and when it does, the RW $\{X_n\}$ is at some height $j$ close to $\zeta k$, so $\{X_n\}$ has a high probability to reach height $(\zeta - 2\beta)k$ within an additional time $\rho b$.  More precisely, for $j \in ((\zeta-2\beta)k,(\zeta+2\beta)k)$ and $l \in [a,(1-2\rho)b]$, provided $\zeta^2/\rho\chi$ is sufficiently small (depending on $\theta$), using \eqref{Mxapprox}, \eqref{exctail}, \eqref{excdecomp2} and our assumption $a \geq \chi k^2$ we have 
\begin{align} \label{finish}
  P_k^*&\big( \tau_{(\zeta-2\beta)k} \in [a,(1-\rho)b] \mid D \cap \{\tau_{\zeta k}^{\BI} = l,X_l= j\} \big) \\
  &= P_j\big( \tau_{(\zeta-2\beta)k} \leq (1-\rho)b - l \big) \notag \\
  &\geq P_j(\tau_{(\zeta-2\beta)k} \leq \rho b \big) \notag \\
  &\geq 1 - M_j P_0(\tau_0 \geq \rho b) \notag \\
  &\geq 1-\theta. \notag
  \end{align}
Since $l,j$ are arbitrary, the same bound holds if we just condition on $D$.  From this and \eqref{BIvsRW1} we get
\begin{align} \label{BIvsRW2}
  P_k^{\BI}&\big( \tau_{\zeta k}^{\BI} \in [a,(1-2\rho)b] \big) \\
  &\leq P_k^*\big( \tau_{(\zeta-2\beta)k} < \tau_{\zeta k}^{\BI} \wedge (1-2\rho)b, N(\tau_{(\zeta-2\beta)k}) \geq \beta k \big) \notag \\
  &\quad + \frac{1}{1-\theta} P_k\big( \tau_{(\zeta-2\beta)k} \in [a,(1-\rho)b] \big). \notag
  \end{align}
Reasoning similarly to \eqref{Sbound} using \eqref{mk2}, and then using \eqref{Bessellower} and \eqref{BIvsBe4}, we get that for some $K_{22}(\zeta,\beta)$ and $K_{23}(\zeta,\beta,\theta,\Dm,\chi)$, for $a \geq K_{23}$,
\begin{align} \label{missteps}
  P_k^*&\big( \tau_{(\zeta-2\beta)k} < \tau_{\zeta k}^{\BI} \wedge (1-2\rho)b, N(\tau_{(\zeta-2\beta)k}) \geq \beta k \big) \\
  &\leq e^{-K_{22}\sqrt{a}} \notag \\
  &\leq \theta P_k^{\BI}\big( \tau_{\zeta k}^{\BI} \in [a,(1-2\rho)b] \big). \notag
  \end{align}
With \eqref{BIvsRW2} this shows that
\begin{align} \label{BevsRW1} 
    &(1-\theta)^2 P_k^{\BI}\big( \tau_{\zeta k}^{\BI} \in [a,(1-2\rho)b] \big)
      \leq P_k\big( \tau_{(\zeta-2\beta)k} \in [a,(1-\rho)b] \big),
  \end{align}
which yields the 4th inequality in \eqref{toprove2}.

For the 5th inequality in \eqref{toprove2}, from \eqref{Mxapprox}, \eqref{exctail}, \eqref{excdecomp2} and \eqref{mk2}, provided $\zeta^2/\rho\chi$ is sufficiently small (depending on $\theta$), we have
\[
  P_{(\zeta-2\beta)k}\big( \tau_0 > \rho b \big) \leq M_{(\zeta-2\beta)k}P_0(\tau_0 \geq \rho b) \leq \theta.
  \]
Hence
\begin{align} \label{0vshigher}
  (1-\theta) &P_k\big( \tau_{(\zeta-2\beta)k} \in [a,(1-\rho)b] \big) \\
  &\leq P_k\big( \tau_{(\zeta-2\beta)k} \in [a,(1-\rho)b] \big) P_{(\zeta-2\beta)k}\big( \tau_0 \leq \rho b \big) \notag \\
  &\leq P_k\big( \tau_0 \in [a,b] \big), \notag
\end{align}
which proves the 5th inequality.

Next, to prove the 7th inequality in \eqref{toprove2}, we can repeat \eqref{BIvsRW1}---\eqref{BevsRW1} with $\{X_n\}$ and $\{X_n^{\BI}\}$ interchanged, and with $\zeta k,(\zeta-2\beta)k$ replaced by $(\zeta+2\beta)k,\zeta k$, respectively, to obtain first the following analog of \eqref{BIvsRW2} and \eqref{missteps}:
\begin{align} \label{BIvsRW3}
  P_k&\big( \tau_{(\zeta+2\beta)k} \in [(1-\rho)a,b] \big) \\
  &\leq P_k^*\big( \tau_{\zeta k}^{\BI} < \tau_{(\zeta+2\beta)k} \wedge b, N(\tau_{\zeta k}^{\BI}) \geq \beta k \big) 
    + \frac{1}{1-\theta} P_k\big( (\tau_{\zeta k}^{\BI} \in [(1-\rho)a,(1+\rho)b] \big) \notag \\
  &\leq \theta P_k\big( \tau_{(\zeta+2\beta)k} \in [(1-\rho)a,b] \big) + \frac{1}{1-\theta} P_k\big( (\tau_{\zeta k}^{\BI} \in [(1-\rho)a,(1+\rho)b] \big), \notag
\end{align}
and from this the analog of \eqref{BevsRW1}:
\begin{align} \label{BevsRW3}
  P_k\big( \tau_{(\zeta+2\beta)k} \in [(1-\rho)a,b] \big)
    &\leq (1+3\theta)P_k^{\BI}\big( \tau_{\zeta k}^{\BI} \in [(1-\rho)a,(1+\rho)b] \big),
  \end{align}
so the 7th inequality is proved.  Here for the second inequality in \eqref{BIvsRW3}, analogously to \eqref{missteps}, we require a lower bound for $P_k\big( \tau_{(\zeta+2\beta)k} \in [(1-\rho)a,b] \big)$, and this follows from \eqref{Bessellower} and the inequality
\[
  P_k\big( \tau_{(\zeta-2\beta)k} \in [a,(1-\rho)b] \big) \geq \frac{1-6\theta}{1-\theta} P_k^{\Be}(\tau_0 \in [a,b])
  \]
which is contained in the first four inequalities of \eqref{toprove2}, with trivial modification to replace $\zeta-2\beta$ with $\zeta+2\beta$ and $[a,(1-\rho)b]$ with $[(1-\rho)a,b]$.

For the 6th inequality, we have
\begin{align} \label{Bevshigher2}
  P_k\big( \tau_0 \in [a,b] \big) \leq P_k\big( &\tau_{(\zeta+2\beta)k} \in [(1-\rho)a,b] \big) \\
  &+ P_k\big( \tau_{(\zeta+2\beta)k} < (1-\rho)a, \tau_0 \in [a,b] \big). \notag
\end{align}
Let us show that the last probability in \eqref{Bevshigher2} is much smaller than the first one.  The Markov property at $\tau_{(\zeta+2\beta)k}$, together with \eqref{Mxapprox}, \eqref{exctail}, \eqref{excdecomp2} and \eqref{mk2}, yields that for some $K_{24}$, provided $a$ is sufficiently large,
\begin{align} \label{smallbit}
   P_k\big( \tau_{(\zeta+2\beta)k} < (1-\rho)a, \tau_0 \in [a,b] \big) &\leq P_{(\zeta+2\beta)k}\big( \tau_0 \geq \rho a \big) \\
   &\leq M_{(\zeta+2\beta)k}P_0(\tau_0 \geq \rho a) \notag \\
   &\leq K_{24}\left( \frac{ \zeta^2k^2 }{ \rho a } \right)^\kappa. \notag
   \end{align}
From \eqref{mk2}, \eqref{Bessellower} and the first half of \eqref{toprove2} we have that for some $K_{25}$,
\[
  P_k\big( \tau_0 \in [a,b] \big) \geq (1-6\theta)P_k^{\Be}(\tau_0 \in [a,b]) \geq K_{25} \Dm 
    e^{-1/2\chi} \left( \frac{k^2}{a} \right)^\kappa.
  \]
From this and \eqref{smallbit} we obtain that provided $\zeta^2/\rho$ is sufficiently small (depending on $\Dm,\theta,\chi$), the ratio of the last to the first probability in \eqref{Bevshigher2} is at most $\theta$, which with \eqref{Bevshigher2} shows that 
\[
  (1-\theta)P_k\big( \tau_0 \in [a,b] \big) \leq P_k\big( \tau_{(\zeta+2\beta)k} \in [(1-\rho)a,b] \big),
  \]
proving the 6th inequality in \eqref{toprove2}, which completes the full proof of \eqref{toprove2}.  Statement \eqref{sandwich7} is then immediate, and then, as we have noted, \eqref{sandwich6} follows.
\end{proof}

Let us now prove Theorem \ref{hittime} for low starting heights---suppose that 
\[
  1 \leq k < \sqrt{\chi m}.
  \]
Let $a \in [m/2,m)$.
We will use Corollary \ref{latticepath2}(ii) and (iii), with $[p,q] = [a/2,a]$, together with \eqref{excpoint}.  By \eqref{excpoint}, provided $\rho$ is sufficiently small (depending on $\theta$) and then $a$ is sufficiently large, we have
\begin{align} \label{latticelower}
  P_0&\left( \tau_0 - \tau_k \in \left[ \frac{a}{2},a \right] \right)) \\
  &\geq P_0\left( \tau_k \leq \rho a, \tau_0 \in \left[ \left( \half + \rho \right)a,a \right] \right) \notag \\
  &= P_0\left( \tau_0 \in \left[ \left( \half + \rho \right)a,a \right] \right) 
    \left[ 1 - P_0\left( \tau_k > \rho a\ \bigg|\ \tau_0 \in \left[ \left( \half + \rho \right)a,a \right] \right) \right] \notag \\
  &\geq (1-\theta) P_0\left( \tau_0 \in \left[ \frac{a}{2},a \right] \right) 
    \left[ 1 - P_0\left( \tau_k > \rho a\ \bigg|\ \tau_0 \in \left[ \left( \half + \rho \right)a,a \right] \right) \right]. \notag
  \end{align}
We need an upper bound for the conditional probability on the right side of \eqref{latticelower}.  For some $K_{26},K_{27}$ we have from \eqref{Mxapprox}, \eqref{exctail}, \eqref{mk2} and Lemma \ref{lowheightlemma} that provided $\chi$ is sufficiently small (depending on $\theta,\rho$),
\begin{align} \label{slowclimb2}
  P_0&\left( \tau_k > \rho a\ \bigg|\ \tau_0 \in \left[ \left( \half + \rho \right)a,a \right] \right) \\
  &\leq \frac{ P_0( \tau_k \wedge \tau_0 > \rho a) }
    { P_0\left( \tau_0 \in \left[ \left( \half + \rho \right)a,a \right] \right) } \notag \\
  &\leq K_{26} \frac{ e^{-K_{4} \rho a/k^2} }{ M_k a^{-\kappa}L(\sqrt{a}) } \notag \\
  &\leq K_{27} \frac{L(k)}{L(\sqrt{a})} \left( \frac{a}{k^2} \right)^\kappa e^{-K_{4} \rho a/k^2} \notag \\
  &\leq \theta. \notag
  \end{align}

Now \eqref{latticelower}, \eqref{slowclimb2}, \eqref{exctail} and \eqref{excpoint} show that
\begin{align} \label{latticelower2}
  P_0\left( \tau_0 - \tau_k \in \left[ \frac{a}{2},a \right] \right)) &\geq (1-2\theta) P_0\left( \tau_0 \in \left[ \frac{a}{2},a \right] \right) \\
  &\geq (1-3\theta) P_0\left( \tau_0 \in \left[ \frac{a}{2}-1,a-1 \right] \right) \notag
  \end{align}
which with Corollary \ref{latticepath2}(ii), (iii) shows that 
\begin{equation} \label{corollaryuse}
  (1-3\theta) M_k P_0(\tau_0 = m) \leq P_k(\tau_0 = m) \leq M_k P_0(\tau_0 = m), \quad m \text{ even},
  \end{equation}
\begin{equation} \label{corollaryuse2}
  (1-3\theta) M_k P_0(\tau_0 = m-1) \leq P_k(\tau_0 = m) \leq M_k P_0(\tau_0 = m-1), \quad m \text{ odd}.
  \end{equation}
This and \eqref{excpoint} prove \eqref{hitapprox}.

Next we prove Theorem \ref{hittime} for midrange starting heights--suppose
\[
  \sqrt{m\chi} \leq k \leq \sqrt{ \frac{m}{\chi} }.
  \]
Let $\theta>0$ and let $0 < \Dm < \DM$ be as in Proposition \ref{intervalmid}.  For the first inequality in \eqref{hitapprox2} we use Corollary \ref{latticepath2}(i) and Proposition \ref{intervalmid}.  From Corollary \ref{latticepath2}(i) and Theorem \ref{tau0tail}, for $m$ large with $m-k$ even, and $0 \leq j < \Dm m$ with $m-j$ even, provided $\Dm$ is small enough (depending on $\theta$), we have
\begin{equation} \label{comparedowne}
  f_m^{(k)} \geq \frac{ f_m^{(0)} }{ f_{m-j}^{(0)} } f_{m-j}^{(k)} \geq (1-\theta) f_{m-j}^{(k)} \quad \text{if $k$ is even},
  \end{equation}
\begin{equation} \label{comparedowno}
  f_m^{(k)} \geq \frac{ f_{m-1}^{(0)} }{ f_{m-j-2}^{(0)} } f_{m-j-1}^{(k)} \geq (1-\theta) f_{m-j-1}^{(k)} \quad \text{if $k$ is odd},
  \end{equation}
so that, averaging over $j$ and applying Proposition \ref{intervalmid}, provided $\Dm$ is small enough (depending on $\chi,\theta$),
\begin{align} \label{averaging}
  P_k(\tau_0 = m) &\geq (1-2\theta) \frac{2}{\Dm m} P_k( (1-\Dm)m < \tau_0 < m) \\
  &\geq (1-3\theta) \frac{2}{\Gamma(k)m} \left( \frac{k^2}{2m} \right)^\kappa e^{-k^2/2m}. \notag
  \end{align}
For the second inequality in \eqref{hitapprox2} the proof is similar:  in place of \eqref{comparedowne} and \eqref{comparedowno} we have that for $m+j$ even,
\begin{equation} \label{compareupe}
  f_m^{(k)} \leq \frac{f_m^{(0)}}{f_{m+j}^{(0)}} f_{m+j}^{(k)} \leq (1+\theta) f_{m+j}^{(k)} \quad \text{if $k$ is even},
  \end{equation}
\begin{equation} \label{compareupo}
    f_m^{(k)} \leq \frac{f_{m-1}^{(0)}}{f_{m+j}^{(0)}} f_{m+j+1}^{(k)} \leq (1+\theta) f_{m+j+1}^{(k)} \quad \text{if $k$ is odd}.
\end{equation}
and then as with \eqref{averaging},
\begin{align} \label{averaging2}
  P_k(\tau_0=m) &\leq  (1+\theta) \frac{2}{\Dm m} P_k\big( \tau_0 \in [m,(1+\Dm)m] \big) \notag \\
  &\leq (1+3\theta)\frac{2}{\Gamma(\kappa)m} \left( \frac{k^2}{2m} \right)^\kappa e^{-k^2/2m},
\end{align}
completing the proof of \eqref{hitapprox2}.

Last, we prove Theorem \ref{hittime} for high starting heights.  We may assume $k \leq m$.  From the first inequalities in \eqref{compareupe} and \eqref{compareupo} and from Theorem \ref{tau0tail}, averaging over $j \in [0,m/8]$ we obtain that for $m$ large and $0<h<k/3$ we have
\begin{align} \label{averaging3}
  P_k(\tau_0=m) &\leq  \left( \frac{9}{8} \right)^\kappa \frac{32}{m} 
    P_k\left( \tau_0 \in \left[m,\frac{9}{8}m\right] \right) \notag \\
  &\leq \left( \frac{9}{8} \right)^\kappa \frac{32}{m} P_k\left( \tau_h \leq \frac{9}{8}m \right).
\end{align}
To bound the last probability we couple our Bessel-like RW to a symmetric simple RW.  Recall that $N(t)$ denotes the number of alarms by time $t$, and let 
\[
  N^* = \left| \left\{i \leq \frac{9}{8}m:  X_i \geq h, X_i^{\sym} \geq h, \text{and an alarm occurs at time } i \right\} \right|.
  \]
Analogously to \eqref{deviation} we have
\begin{align} \label{twoways}
  P_k\left( \tau_h \leq \frac{9}{8}m \right) &= P_k^*\left( \tau_{3h}^{\sym} > \tau_h, \tau_h \leq \frac{9}{8}m \right) 
    + P_k^*\left( \tau_{3h}^{\sym} \leq \tau_h \leq \frac{9}{8}m \right) \notag \\
  &\leq P_k^*\left( N^* > h \right) + P_k^{\sym}\left( \tau_{3h}^{\sym} \leq \frac{9}{8}m \right).
\end{align}
We now take $h= k/8$; we assume for convenience that $h$ is an integer.   If $m_0$ (and hence $k$) is large enough, then $\sup_{x \geq k/8} |p_x - \half| \leq 2(1+|\delta|)/k$.  Then $N^*$ is stochastically smaller than a Binomial($9m/8,2(1+|\delta|)/k$) random variable.  We apply Bennett's Inequality (see Hoeffding \cite{Ho63}), which states that for a Binomial($n,p$) random variable $Y$ and $\lambda > np$,
\[
  P(Y \geq \lambda) \leq e^{-\lambda\psi(np(1-p)/\lambda)},
  \]
where $\psi$ is the decreasing function
\[
  \psi(x) = (1+x)\log\left( 1 + \frac{1}{x} \right) - 1.
  \]
For $x \leq 1/4$ we have $\psi(x) \geq 1$ and hence $\psi(x) \geq \half (1+x)\log\left( 1 + \frac{1}{x} \right)\geq \half \log\frac{1}{x}$.  Therefore for $\lambda \geq 4np$, 
\[
  P(Y \geq \lambda) \leq e^{-(\lambda/2) \log(\lambda/np)},
  \]
and in particular, provided $\chi$ is sufficiently small we have
\begin{equation} \label{binombound}
  P_k^*\left( N^* > \frac{k}{8} \right) \leq e^{-(k/16) \log( k^2/18(1+|\delta|) m)} \leq e^{-k/4} \leq e^{-k^2/4m}.
  \end{equation}
Also, again provided $\chi$ is small, by Hoeffding's inequality \cite{Ho63},
\[
  P_k^{\sym}\left( \tau_{k/4}^{\sym} \leq \frac{9}{8}m \right) \leq 2P_0^{\sym}\left(X_{9m/8}^{\sym} > \frac{3}{4}k \right)
    \leq e^{-k^2/4m},
  \]
which with \eqref{averaging3}, \eqref{twoways} and \eqref{binombound} yields
\begin{equation} \label{highconcl}
  P_k(\tau_0=m) \leq  \left( \frac{9}{8} \right)^\kappa \frac{64}{m} e^{-k^2/4m} \leq \frac{1}{m} e^{-k^2/8m},
  \end{equation}
completing the proof of \eqref{hitapprox3}.\\[10 pt]

\section{Proof of Theorems \ref{distrib} and \ref{location}}

Theorem \ref{location} is a straightforward consequence of Theorem \ref{distrib}, \eqref{reverse} and \eqref{lambda}, so we prove Theorem \ref{distrib}.
We use \eqref{renewal}, applying Theorem \ref{hittime} and \eqref{finitemean3}---\eqref{c2case3} to approximate the products on the right side.  

We consider first part (i), for low starting heights, i.e. $1 \leq k < \sqrt{\chi n}$.  By \eqref{exctail}, \eqref{excpoint} and \eqref{renewal} and Theorem \ref{hittime}, given $\theta>0$, taking $\rho$ and then $\chi$ sufficiently small, for $n$ large, we have the following sandwich bound for $P_k(X_n = 0)$:
\begin{align} \label{lowstart}
  (1&-2\theta) P_0(X_{\tilde{n}} = 0) \\
  &\leq (1-\theta) \min_{(1-\rho)n \leq j \leq n} P_0(X_{\tilde{j}} = 0) \notag \\
  &\leq P_k(\tau_0 \leq \rho n) \min_{(1-\rho)n \leq j \leq n} P_0(X_{\tilde{j}} = 0) \notag \\
  &\leq P_k(X_n = 0) \notag \\
  &\leq \max_{(1-\rho)n \leq j \leq n} P_0(X_{\tilde{j}} = 0) + \sum_{0 \leq j < (1-\rho)n} P_k(\tau_0 = n-j) P_0(X_j = 0) \notag \\
  &\leq (1+\theta) P_0(X_{\tilde{n}} = 0) + \max_{0 \leq i < (1-\rho)n} P_k(\tau_0 = n-i)  \sum_{0 \leq j < (1-\rho)n} P_0(X_j = 0) \notag \\
  &\leq (1+\theta) P_0(X_{\tilde{n}} = 0) + K_{28} (\rho n)^{-(\kappa+1)}L(\sqrt{n}) M_{\sqrt{\chi n}}  \sum_{0 \leq j < (1-\rho)n} P_0(X_j = 0). \notag 
  \end{align}
We need to show that the second term on the right side of \eqref{lowstart} is small compared to the first term on the right side.  From \eqref{finitemean3}---\eqref{c2case3} we see that for some $K_{29}$, in all three cases, the sum in that second term is bounded by $K_{29}n P_0(X_{\tilde{n}} = 0)$.  Therefore, using \eqref{Mxapprox} and \eqref{excpoint}, if $\chi$ is sufficiently small (depending on $\theta,\rho$) then for large $n$, the second term is bounded above by
\[
  2K_{28}K_{29} \frac{ \chi^\kappa }{ \rho^{\kappa + 1} } P_0(X_{\tilde{n}} = 0) \leq \theta P_0(X_{\tilde{n}} = 0).
  \]
With \eqref{lowstart} this gives
\begin{equation} \label{lowstart1}
  (1-2\theta) P_0(X_{\tilde{n}} = 0) \leq P_k(X_n = 0) \leq (1+2\theta) P_0(X_{\tilde{n}} = 0),
  \end{equation}
as desired.

Next we consider part (ii), for $E_0(\tau_0)<\infty$ and midrange starting heights, $\sqrt{n\chi} \leq k \leq \sqrt{n/\chi}$.  By \eqref{finitemean3} there exists $n_1$ such that 
\[
  \frac{2-\theta}{E_0(\tau_0)} \leq P_0(X_n = 0) \leq \frac{2+\theta}{E_0(\tau_0)} \quad \text{for all even } n \geq n_1.
  \]
Let $0 < \tc < \chi$.  Then using \eqref{renewal} and Theorem \ref{hittime} (with $\tc$ in place of $\chi$), provided $\tc$ is sufficiently small, and then $n$ (and hence $k$) is sufficiently large,
\begin{align} \label{midstart1}
  P_k(X_n=0) &\geq \sum_{j=n_1}^{n-\tc k^2} P_k(\tau_0 = n-j) P_0(X_j = 0) \\
  &\geq \frac{2-\theta}{E_0(\tau_0)} \sum_{m:\tc k^2 \leq m \leq n-n_1,m-k\text{ even}} P_k(\tau_0=m) \notag \\
  &\geq \frac{2-3\theta}{E_0(\tau_0)} \sum_{m:\tc k^2 \leq m \leq n-n_1,m-k\text{ even}}
    \frac{2}{\Gamma(\kappa)m} \left( \frac{k^2}{2m} \right)^\kappa e^{-k^2/2m} \notag \\
  &\geq \frac{2-4\theta}{E_0(\tau_0)} \int_{\tc k^2}^{n-n_1} \frac{1}{\Gamma(\kappa)x} \left( \frac{k^2}{2x} \right)^\kappa e^{-k^2/2x}\ dx \notag \\
  &= \frac{2-4\theta}{E_0(\tau_0)} \int_{k^2/2(n-n_1)}^{1/2\tc} \frac{1}{\Gamma(\kappa)} u^{\kappa-1} e^{-u}\ du \notag \\
  &\geq \frac{2-5\theta}{E_0(\tau_0)} \int_{k^2/2n}^\infty \frac{1}{\Gamma(\kappa)} u^{\kappa-1} e^{-u}\ du. \notag
  \end{align}
In the other direction, we have similarly
\begin{equation} \label{midstart2}
  \sum_{j=n_1}^{n-\tc k^2} P_k(\tau_0 = n-j) P_0(X_j = 0)
     \leq \frac{2+5\theta}{E_0(\tau_0)} \int_{k^2/2n}^\infty \frac{1}{\Gamma(\kappa)} u^{\kappa-1} e^{-u}\ du.
  \end{equation}
Also similarly to \eqref{midstart1}, given $\alpha>0$ we have for sufficiently small $\tc$ that provided $n$ is large,
\begin{equation} \label{mostprob}
  P_k(\tc k^2 \leq \tau_0 \leq k^2/\tc) \geq (1-\alpha) \int_{\tc/2}^{1/2\tc} \frac{1}{\Gamma(\kappa)} u^{\kappa-1} e^{-u}\ du \geq 1-2\alpha,
  \end{equation}
so in particular, for small $\tc$,
\[
  P_k(\tau_0 < \tc k^2) \leq \theta \int_{1/2\chi}^\infty \frac{1}{\Gamma(\kappa)} u^{\kappa-1} e^{-u}\ du.
  \]
With \eqref{finitemean3}, \eqref{renewal}, \eqref{midstart2} and Theorem \ref{hittime} this gives
\begin{align} \label{midstart3}
  P_k(X_n=0) &\leq \frac{2+5\theta}{E_0(\tau_0)} \int_{k^2/2n}^\infty \frac{1}{\Gamma(\kappa)} u^{\kappa-1} e^{-u}\ du \\
  &\qquad + P_k(n-n_1 < \tau_0 \leq n) + P_k(\tau_0 < \tc k^2) \max_{n-\tc k^2/2 < j \leq n} P_0(X_j = 0) \notag \\
  &\leq \frac{2+5\theta}{E_0(\tau_0)} \int_{k^2/2n}^\infty \frac{1}{\Gamma(\kappa)} u^{\kappa-1} e^{-u}\ du \notag \\
  &\qquad + \frac{K_{30}n_1}{n} \left( \frac{k^2}{2n} \right)^\kappa e^{-k^2/2n} 
    + \frac{3\theta}{E_0(\tau_0)} \int_{1/2\chi}^\infty \frac{1}{\Gamma(\kappa)} u^{\kappa-1} e^{-u}\ du \notag \\
  &\leq \frac{2+9\theta}{E_0(\tau_0)} \int_{k^2/2n}^\infty \frac{1}{\Gamma(\kappa)} u^{\kappa-1} e^{-u}\ du, \notag
  \end{align}
which with \eqref{midstart1} proves Theorem \ref{distrib}(ii) for midrange starting heights.

Now consider part (ii) for high starting heights, $k > \sqrt{n/\chi}$.  We may assume $\theta<1$.  Analogously to \eqref{midstart2} we have using \eqref{renewal} and Theorem \ref{hittime} that
\begin{align} \label{highstart1}
  \sum_{j=n_1+1}^{n-k} P_k(\tau_0 = n-j) P_0(X_j = 0) 
    &\leq \frac{3}{E_0(\tau_0)} \sum_{m:k \leq m \leq n-n_1-1,m-k\text{ even}}
    \frac{1}{m} e^{-k^2/8m} \notag \\
  &\leq \frac{3}{E_0(\tau_0)} \frac{k+1}{k} \int_k^{n-n_1} \frac{1}{x} e^{-k^2/8x}\ dx \notag \\
  &= \frac{3}{E_0(\tau_0)} \frac{k+1}{k} \int_{k^2/8(n-n_1)}^{k/8} e^{-u}\ du \notag \\
  &\leq  \frac{4}{E_0(\tau_0)} e^{-k^2/8n}.
  \end{align}
Further, as in \eqref{lowstart}, using Theorem \ref{hittime},
\begin{align} \label{highstart2}
  \sum_{j=0}^{n_1} P_k(\tau_0 = n-j) P_0(X_j = 0) 
    &\leq \left( \max_{j \leq n_1} P_k(\tau_0 = n-j) \right) \sum_{j=0}^{n_1-1} P_0(X_j = 0) \notag \\
  &\leq \frac{n_1+1}{n} e^{-k^2/8n}.
  \end{align}
Now \eqref{renewal}, \eqref{highstart1} and \eqref{highstart2} prove \eqref{finitemean4}.

We turn now to part (iii), for $-1<\delta<1$ and midrange starting heights $\sqrt{n\chi} \leq k \leq \sqrt{n/\chi}$.  
We use the fact that
\begin{equation} \label{Psik}
  \int_0^1 \frac{1}{(1-u)^{1+\kappa} u^{1-\kappa}} e^{-a/(1-u)}\ du = \Gamma(\kappa) a^{-\kappa} e^{-a} \quad
     \text{for all } a, \kappa > 0,
  \end{equation}
as can easily be seen via the change of variable $v = (1-u)^{-1}$.
By \eqref{infinitemean3} there exists $n_2=n_2(\theta)$ such that
\begin{equation} \label{returnzero}
  \frac{(1-\theta)2^\kappa K_0}{\Gamma(1-\kappa)} n^{-(1-\kappa)} L(\sqrt{n})^{-1} \leq P_0(X_n = 0) 
    \leq \frac{(1+\theta)2^\kappa K_0}{\Gamma(1-\kappa)} n^{-(1-\kappa)} L(\sqrt{n})^{-1}  
  \end{equation}
for all even $n \geq n_2$.  Analogously to \eqref{midstart1}, provided $\tilde{\chi}/\chi$ is sufficiently small, using \eqref{renewal}, \eqref{Psik} and \eqref{returnzero} we then obtain that for large $n$,
\begin{align} \label{midstart4}
  P_k&(X_n=0) \\
  &\geq \sum_{j=n_2}^{n-\tc k^2} P_k(\tau_0 = n-j) P_0(X_j = 0) \notag \\
  &\geq \frac{(2-\theta)2^\kappa K_0}{\Gamma(\kappa)\Gamma(1-\kappa)} \sum_{n_2 \leq j \leq n-\tc k^2 \atop j\text{ even}} \frac{1}{n-j}
    \left( \frac{k^2}{2(n-j)} \right)^\kappa e^{-k^2/2(n-j)} j^{-(1-\kappa)}L(\sqrt{j})^{-1} \notag \\
  &\geq \frac{(1-\theta)2^\kappa K_0}{\Gamma(\kappa)\Gamma(1-\kappa)} L(\sqrt{n})^{-1}
    \int_{n_2}^{n-\tc k^2} \frac{1}{n-x} \left( \frac{k^2}{2(n-x)} \right)^\kappa e^{-k^2/2(n-x)} \frac{1}{x^{1-\kappa}}\ dx \notag \\
  &=  \frac{(1-\theta)2^\kappa K_0}{\Gamma(\kappa)\Gamma(1-\kappa)} n^{-(1-\kappa)} L(\sqrt{n})^{-1} \left( \frac{k^2}{2n} \right)^\kappa
    \int_{n_2/n}^{1-\tc k^2/n} \frac{1}{(1-u)^{1+\kappa} u^{1-\kappa}} e^{-k^2/2n(1-u)}\ du \notag \\
  &\geq \frac{(1-2\theta)2^\kappa K_0}{\Gamma(1-\kappa)} n^{-(1-\kappa)} L(\sqrt{n})^{-1} e^{-k^2/2n}. \notag
  \end{align}
In the other direction, analogously to \eqref{midstart2}, from a calculation similar to \eqref{midstart4} we get
\begin{equation} \label{midstart5}
  \sum_{j=n_2}^{n-\tc k^2} P_k(\tau_0 = n-j) P_0(X_j = 0) \leq 
    \frac{(1+2\theta)2^\kappa K_0}{\Gamma(1-\kappa)} n^{-(1-\kappa)} L(\sqrt{n})^{-1} e^{-k^2/2n}.
  \end{equation}
With \eqref{renewal}, \eqref{infinitemean3} and Theorem \ref{hittime} this gives the analog of \eqref{midstart3}:  provided $\tc$ is taken sufficiently small and then $n$ sufficiently large,
\begin{align} \label{midstart6}
  P_k&(X_n=0) \\
  &\leq \frac{(1+2\theta)2^\kappa K_0}{\Gamma(1-\kappa)} n^{-(1-\kappa)} L(\sqrt{n})^{-1} e^{-k^2/2n} \notag \\
  &\qquad + n_2 \max_{0 \leq j < n_2} P_k(\tau_0 = n-j) + P_k(\tau_0 < \tc k^2) \max_{n-\tc k^2 < j \leq n} P_0(X_j=0) \notag \\
  &\leq \frac{(1+2\theta)2^\kappa K_0}{\Gamma(1-\kappa)} n^{-(1-\kappa)} L(\sqrt{n})^{-1} e^{-k^2/2n} \notag \\
  &\qquad + K_{31}(\chi)  n^{-1} 
    + \theta  \frac{(1+2\theta)2^\kappa K_0}{\Gamma(1-\kappa)} n^{-(1-\kappa)} L(\sqrt{n})^{-1} \notag \\
   &\leq \frac{(1+4\theta)2^\kappa K_0}{\Gamma(1-\kappa)} n^{-(1-\kappa)} L(\sqrt{n})^{-1} e^{-k^2/2n}. \notag
\end{align}
Here the second inequality uses the fact that by \eqref{mostprob}, we can make $P_k(\tau_0 < \tc k^2)$ as small as desired by taking $\tc$ small.  Together \eqref{midstart4} and \eqref{midstart6} prove Theorem \ref{distrib}(iii) for midrange starting heights.

We turn next to part (iii) for high starting heights, $k > \sqrt{n/\chi}$.  There exists $K_{32}$ such that for $0<\alpha \leq K_{32}^2$,
\begin{equation} \label{decaysum}
  \sum_{j=1}^\infty e^{-\alpha j} \frac{1}{ j^{1-\kappa}L(\sqrt{j}) } \leq \frac{2}{\kappa} \alpha^{-\kappa} L\left( \frac{1}{\sqrt{\alpha}} \right)^{-1}.
  \end{equation}
Then analogously to \eqref{midstart1} and \eqref{midstart4}, when $k \leq K_{32}n$, using \eqref{renewal}, \eqref{returnzero} and Theorem \ref{hittime} we have for large $n$,
\begin{align} \label{highstart3}
  \sum_{j=n_2}^{n-k} &P_k(\tau_0 = n-j) P_0(X_j = 0) \notag \\
  &\leq \frac{2^{1+\kappa} K_0}{\Gamma(1-\kappa)} \sum_{j=n_2}^{n-k} 
    \frac{1}{n-j} e^{-k^2/8(n-j)} \frac{1}{ j^{1-\kappa}L(\sqrt{j}) } \notag \\
  &\leq \frac{2^{2+\kappa} K_0}{\Gamma(1-\kappa)} \bigg[ \frac{2}{n} e^{-k^2/8n} \sum_{n_2\leq j\leq n/2} e^{-k^2j/8n^2}
    \frac{1}{ j^{1-\kappa}L(\sqrt{j}) } \notag \\
  &\qquad \qquad \qquad +  \sum_{n/2<j\leq n-k} 
    \frac{1}{n-j} e^{-k^2/8(n-j)} \frac{1}{ j^{1-\kappa}L(\sqrt{j}) }\bigg] \notag \\
  &\leq \frac{2^{2+\kappa} K_0}{\Gamma(1-\kappa)} \bigg[ \frac{4\cdot 4^\kappa}{\kappa n} e^{-k^2/8n} \left( \frac{n}{k} \right)^{2\kappa}
    L\left( \frac{n}{k} \right)^{-1} +  \frac{2}{n} e^{-k^2/4n} \sum_{j=1}^n 
    \frac{1}{ j^{1-\kappa}L(\sqrt{j}) }\bigg] \notag \\
  &\leq \frac{2^{2+\kappa} K_0}{\Gamma(1-\kappa)} \bigg[ \frac{4\cdot 4^\kappa}{\kappa n} e^{-k^2/8n} \left( \frac{n}{k} \right)^{2\kappa}
    L\left( \frac{n}{k} \right)^{-1} +  \frac{4}{\kappa n} e^{-k^2/4n} n^\kappa L(\sqrt{n})^{-1} \bigg].
\end{align}
Here in the third inequality we used the fact that $(n-j)^{-1} e^{-k^2/8(n-j)}$ is a decreasing function of $j$, and the fact that $L$ is slowly varying.  Provided $\chi$ is sufficiently small, the second term inside the brackets on the right side of \eqref{highstart3} is smaller than the first term; using this, \eqref{renewal} and Theorem \ref{hittime} we obtain that for some $K_{33}(\kappa)$, provided $\chi$ is small enough,
\begin{align} \label{highstart4}
  P_k(X_n = 0) &\leq \frac{2^{4+3\kappa} K_0K_{32}}{\kappa\Gamma(1-\kappa)n} e^{-k^2/8n} \left( \frac{n}{k} \right)^{2\kappa}
    L\left( \frac{n}{k} \right)^{-1} + \sum_{0 \leq j < n_2} P_k(\tau_0 = n-j) \notag \\
  &\leq \frac{2^{4+3\kappa} K_0K_{32}}{\kappa\Gamma(1-\kappa)n} e^{-k^2/8n} \left( \frac{n}{k} \right)^{2\kappa}
    L\left( \frac{n}{k} \right)^{-1} + \frac{2n_2}{n} e^{-k^2/8n} \notag \\
  &\leq K_{33} e^{-k^2/8n} n^{-(1-\kappa)} L(\sqrt{n})^{-1}.
\end{align}
Here $n_2=n_2(1)$.  This proves \eqref{infinitemean3a} when $k \leq K_{32}n$.  If $K_{32}n<k\leq n$, in place of \eqref{highstart3} and \eqref{highstart4} we have using \eqref{hitapprox3} that
\begin{align} \label{highstart7}
  P_k(X_n = 0) &= \sum_{j=0}^{n-k} P_k(\tau_0 = n-j) P_0(X_j = 0) \notag \\
  &\leq \sum_{j=0}^{n-k} \frac{1}{n-j} e^{-k^2/8n} \notag \\
  &\leq K_{34} e^{-k^2/8n},
\end{align}
from which \eqref{infinitemean3a} follows.

Next we consider part (iv), in which $\delta=1$, $E_0(\tau_0) = \infty$, in the case of midrange starting heights $\sqrt{n\chi} \leq k \leq \sqrt{n/\chi}$.  In this case $\mu_0$ is slowly varying, and by \eqref{c2case3} there exists $n_3=n_3(\theta)$ such that 
\begin{equation} \label{returnzero2}
  \frac{2-\theta}{\mu_0(n)} \leq P_0(X_n = 0) \leq \frac{2+\theta}{\mu_0(n)} \quad \text{for all even } n \geq n_3.
  \end{equation}
Then analogously to \eqref{midstart1}, using Theorem \ref{hittime}, \eqref{renewal} and \eqref{returnzero2}, for large $n$,
\begin{align} \label{midstart7}
  P_k(X_n = 0)  &\geq \sum_{j=n_3}^{n-\tc k^2} P_k(\tau_0 = n-j) P_0(X_j = 0) \\
  &\geq \frac{4-3\theta}{\Gamma(\kappa)} \sum_{\tc k^2 \leq m \leq n-n_3 \atop n-m\text{ even}} \frac{1}{m}
    \left( \frac{k^2}{2m} \right)^\kappa e^{-k^2/2m} \frac{1}{\mu_0(n-m)} \notag \\
  &\geq \frac{2-2\theta}{\mu_0(n)} \int_{k^2/2n}^\infty \frac{1}{\Gamma(\kappa)} u^{\kappa-1} e^{-u}\ du, \notag
\end{align}
and similarly
\begin{equation} \label{midstart8}
  \sum_{j=n_3}^{n-\tc k^2} P_k(\tau_0 = n-j) P_0(X_j = 0) 
    \leq \frac{2+2\theta}{\mu_0(n)} \int_{k^2/2n}^\infty \frac{1}{\Gamma(\kappa)} u^{\kappa-1} e^{-u}\ du.
  \end{equation}
Using \eqref{renewal}, \eqref{returnzero2}, \eqref{midstart8} and Theorem \ref{hittime}, and taking $\tilde{\chi}$ sufficiently small, we obtain the analog of \eqref{midstart6}:
\begin{align} \label{midstart9}
  P_k&(X_n = 0) \\
  &\leq \frac{2+2\theta}{\mu_0(n)} \int_{k^2/2n}^\infty \frac{1}{\Gamma(\kappa)} u^{\kappa-1} e^{-u}\ du \notag \\
  &\qquad + n_3 \max_{0 \leq j < n_3} P_k(\tau_0 = n-j) + P_k(\tau_0 < \tc k^2) \max_{n-\tc k^2 < j \leq n} P_0(X_j=0) \notag \\
  &\leq \frac{2+2\theta}{\mu_0(n)} \int_{k^2/2n}^\infty \frac{1}{\Gamma(\kappa)} u^{\kappa-1} e^{-u}\ du \notag \\
  &\qquad + K_{35}(\chi) n^{-1} 
    + \frac{\theta}{\mu_0(n)} \int_{1/2\chi}^\infty \frac{1}{\Gamma(\kappa)} u^{\kappa-1} e^{-u}\ du \notag \\
  &\leq \frac{2+4\theta}{\mu_0(n)} \int_{k^2/2n}^\infty \frac{1}{\Gamma(\kappa)} u^{\kappa-1} e^{-u}\ du. \notag
\end{align}
Together \eqref{midstart7} and \eqref{midstart9} prove Theorem \ref{distrib}(iv) for midrange starting heights.

Last we consider part (iv) for high starting heights, $k > \sqrt{n/\chi}$. We may assume $\theta<1$.  When $k \leq K_{32}n$ and $k$ is sufficiently large, we have analogously to \eqref{highstart3}, using \eqref{decaysum} and Theorem \ref{hittime},
\begin{align} \label{highstart5}
  \sum_{j=n_3}^{n-k} &P_k(\tau_0 = n-j) P_0(X_j = 0) \notag \\
  &\leq \sum_{j=n_3}^{n-k} \frac{1}{n-j} e^{-k^2/8(n-j)}\frac{3}{\mu_0(j)} \notag \\
  &\leq \frac{6}{n} e^{-k^2/8n} \sum_{n_3\leq j\leq n/2} e^{-k^2j/8n^2}
    \frac{1}{ \mu_0(j) } + \frac{6}{n} e^{-k^2/4n}  \sum_{n/2<j\leq n-k} \frac{1}{ \mu_0(j) } \notag \\
  &\leq \frac{96}{n} e^{-k^2/8n}  \frac{n^2}{k^2} \mu_0\left( \frac{n^2}{k^2} \right)^{-1} + e^{-k^2/4n} \frac{6}{\mu_0(n)} \notag \\
  &\leq e^{-k^2/8n} \frac{7}{\mu_0(n)}.
\end{align}
In the last inequality we have bounded $(n^2/k^2)\mu_0(n^2/k^2)^{-1}$ by $n/96\mu_0(n)$, valid for $\chi$ sufficiently small because $n^2/k^2 \leq \chi^2 n$ and $\mu_0$ is slowly varying.  Then using \eqref{renewal} and \eqref{hitapprox3},
\begin{align} \label{highstart6}
  P_k(X_n = 0) &\leq e^{-k^2/8n} \frac{7}{\mu_0(n)} + \sum_{1 \leq j < n_3} P_k(\tau_0 = n-j) \notag \\
  &\leq e^{-k^2/8n} \frac{7}{\mu_0(n)} + \frac{n_3}{n} e^{-k^2/8n} \notag \\
  &\leq e^{-k^2/8n} \frac{8}{\mu_0(n)}.
  \end{align}
If instead $K_{32}n<k\leq n$, then \eqref{highstart7} is valid.  In fact, a look at \eqref{highconcl} shows that, by reducing $\chi$ if necessary, we can replace 8 on the right side of \eqref{highstart7} with any constant greater than 4.  Therefore in place of \eqref{highstart6} we have for large $n$ that
\begin{align} \label{highstart8}
  P_k(X_n = 0) \leq K_{34} e^{-k^2/6n} \leq \frac{1}{\mu_0(n)} e^{-k^2/8n}.
\end{align}
Thus \eqref{borderline3} holds in both cases.

\section{Acknowledgements}

The author thanks P. Baxendale, F. Dunlop, T. Huillet, J. Pitman and A. Wade for helpful discussions, and thanks a referee for multiple references and suggestions.

\end{document}